\documentclass[final,notitlepage,12pt,reqno,tbtags]{amsart}
\usepackage{graphicx}
\usepackage{calc}
\usepackage{url}
\usepackage{mathrsfs}
\newlength{\depthofsumsign}
\makeatletter
\let\I\@undefined
\makeatother
\usepackage{geometry}
\usepackage{indentfirst}
\usepackage[normalem]{ulem}
\usepackage{float}
\usepackage{amsthm}

\usepackage{enumitem}[2011/09/28]
\setenumerate{align=left, leftmargin=0pt,labelsep=.5em, labelindent=0\parindent,listparindent=\parindent,itemindent=*}
\usepackage{array}
\usepackage{empheq}
\usepackage{natbib}
\usepackage[all]{xy}

\usepackage{caption}[2012/02/19]

\usepackage{longtable,lscape}

\usepackage{fouriernc}
\usepackage{fourier}
\usepackage{amssymb,bm,amsmath}
\usepackage{amsfonts}

\usepackage[OT2,T1,T2A]{fontenc}

\usepackage[english,french,german,russian]{babel}
\usepackage{appendix}

\DeclareMathOperator{\IKM}{\mathbf{IKM}}

\DeclareMathOperator{\JYM}{\mathbf{JYM}}

\DeclareMathOperator{\D}{d}
\DeclareMathOperator{\I}{Im}
\DeclareMathOperator{\R}{Re}

\bibpunct{[}{]}{,}{n}{}{;}

\def\eor{\hfill$ \square$}

\theoremstyle{plain}
\newtheorem{theorem}{Theorem}[subsection]
\newtheorem{proposition}[theorem]{Proposition}
\newtheorem{lemma}[theorem]{Lemma}

\newenvironment{remark}[1][Remark]{\begin{trivlist}
\item[\hskip \labelsep {\bfseries #1}]}{\end{trivlist}}

\newbox\shell
\newcommand{\dia}[2]{\setbox\shell=\hbox{\begin{picture}(180,120)(-90,-60)#1
\put(-90,-60){\makebox(180,120)[b]{\large #2}}\end{picture}}\dimen0=\ht
\shell\multiply\dimen0by7\divide\dimen0by16\raise-\dimen0\box\shell\hfill}

\theoremstyle{definition}

\numberwithin{equation}{subsection}

\usepackage{color}

\begin{document}

\pagenumbering{roman}
\selectlanguage{english}
\title[Wick rotations, Eichler integrals and multi-loop Feynman diagrams]{Wick Rotations, Eichler integrals,\\ and multi-loop Feynman diagrams}
\author{Yajun Zhou}
\address{Program in Applied and Computational Mathematics (PACM), Princeton University, Princeton, NJ 08544; Academy of Advanced Interdisciplinary Studies (AAIS), Peking University, Beijing 100871, P. R. China }
\email{yajunz@math.princeton.edu, yajun.zhou.1982@pku.edu.cn}
\date{\today}

\maketitle

\begin{abstract}
    Using contour deformations and integrations over modular forms, we compute certain Bessel moments arising from diagrammatic expansions in two-dimensional quantum field theory. We evaluate these Feynman integrals  as either explicit constants or  critical values of modular $L$-series, and verify several recent conjectures of Broad\-hurst. \end{abstract}

\pagenumbering{roman}
\tableofcontents

\clearpage

\pagenumbering{arabic}

\section{Introduction}
\subsection{Background and motivations}
In quantum field theory (QFT), we encounter integrals over Bessel functions while performing diagrammatic expansions in the configuration space. For two-dimensional QFT,   we need Bessel functions $ J_0$ and $Y_0$, as well as modified Bessel functions $I_0$ and $K_0$, to define propagators and compute Feynman integrals \cite{Groote2007,Bailey2008,Broadhurst2013MZV,Broadhurst2016,BroadhurstMellit2016}.

 We are interested in    Bessel moments    $ \JYM(\alpha,\beta;\nu):=\int_0^\infty[J_0(t)]^\alpha[Y_0(t)]^{\beta}t^{\nu}\D t$  and   $ \IKM(a,b;n):=\int_0^\infty[I_0(t)]^a[K_0(t)]^{b}t^{n}\D t$, where the non-negative integers $ \alpha,\beta,\nu,a,b,n$ are chosen to ensure convergence of the corresponding integrals. The Bessel moments $ \JYM$'s are  useful auxiliary tools for computing $\IKM$'s in  two-dimensional QFT. Furthermore, the $ \IKM$'s also show up in the finite part for  renormalized perturbative expansions of four-dimensional QFT: for example, $\IKM(1,5;1) $ and $ \IKM(1,5;3)$ are part of the  4-loop contributions  (from 891 Feynman diagrams) to  electron's magnetic moment \cite[][(19) and Fig.~3($a$)($a'$)]{Laporta:2017okg}, according to the standard formulation of quantum electrodynamics (four-dimensional QFT).

The mathematical understanding of
$\JYM(\alpha,\beta;\nu)$ for $ \alpha+\beta\geq5$  and
$\IKM(a,b;n)$ for $ a+b\geq5$ is relatively scant.
While numerical experiments have suggested a rich collection of identities relating various cases of
$\IKM(a,b;1)$ (each of which corresponding to a Feynman diagram containing $b-1$ loops) to special values of certain Hasse--Weil $L$-series  for $ a+b\in\{5,6,7,8\}$ \cite{Broadhurst2013MZV,Broadhurst2016,BroadhurstMellit2016}, most of these conjectural evaluations are heretofore unproven.

In our recent work \cite{HB1}, we have  shown that \begin{align}
\int_0^\infty \frac{[\pi I_{0}(t)+iK_0(t)]^m+[\pi I_{0}(t)-iK_0(t)]^m}{i}[K_0(t)]^mt^n\D t=0\label{eq:HB_sum_rule}
\end{align} for  $ m\in\mathbb Z_{>1},n\in\mathbb Z_{\geq0},\frac{m-n}{2}\in\mathbb Z_{>0}$,   and \begin{align}\label{eq:HB_sum_rule'}
\int_0^\infty \frac{[\pi I_{0}(t)+iK_0(t)]^m-[\pi I_{0}(t)-iK_0(t)]^m}{i}[K_0(t)]^mt^n\D t=0
\end{align}for $ m\in\mathbb Z_{>0},n\in\mathbb Z_{\geq0},\frac{m-n-1}{2}\in\mathbb Z_{>0}$  (Bai\-ley--Bor\-wein--Broad\-hurst--Glasser sum rule  {\cite[][``final conjecture'', (220)]{Bailey2008}}, with  generalizations). In addition, we have also confirmed that \begin{align}\begin{split}&\frac{2^{1+2(n-1)[1-(-1)^{m}]}}{\pi^{m+1}}\int_0^\infty \frac{[\pi I_0(t)+i K_{0}(t)]^{m}-[\pi I_0(t)-i K_{0}(t)]^{m}}{i}\times\\&\times[K_0(t)]^{m}(2t)^{2n+m-3}\D t\end{split}\end{align}  evaluates to a positive integer for all $ m,n\in\mathbb Z_{>0}$ (Broadhurst--Mellit integer sequence  {\cite[][(149) in Conjecture 5]{Broadhurst2016}} and Broadhurst--Roberts rational sequence \cite[][Conjecture 2]{Broadhurst2017Paris}). While the aforementioned results resolve some longstanding conjectures, they barely scratch the surface of the algebraic and arithmetic nature of Bessel moments. For example, the determinant    $ \IKM(1,4;1)\IKM(2,$ $3;3)-\IKM(2,3;1)\IKM(1,4;3)=2\pi^{3}/\sqrt{3^35^5}$ \cite[conjectured in][(100)]{Broadhurst2016} and the sum rule     $ 9\pi^2\IKM(4,$ $4;1)-14\IKM(2,6;1)=0$   \cite[conjectured in][(147)]{Broadhurst2016} had not been covered by the real-analytic methods we employed in  \cite{HB1}.

\subsection{Statement of results and plan of proof}
 In this  article, we supplement our previous work with complex analysis and modular forms, which are two powerful devices that not only produce new algebraic relations among different $ \IKM$ moments, but also connect Feynman diagrams to special $L$-values and Kluyver's ``random walk integrals'' $\JYM(n,0,1) ,n\in\mathbb Z_{\geq5}$ \cite{Kluyver1906,BSWZ2012,BSV2016}.

The layout of this paper is described in the next four paragraphs.

Beginning with a brief survey of the analytic properties for (modified) Bessel functions in \S\ref{subsec:HankelH1H2}, we
introduce Wick rotations, which are contour deformations that allow us to convert $ \IKM$ problems into $\JYM$ problems, in \S\ref{subsec:Wick}. We demonstrate the usefulness of Wick rotations by a very short (yet self-contained) proof of the closed-form evaluation of a Bessel moment \begin{align}
\int_{0}^\infty I_0(t)[K_0(t)]^4t\D t=\frac{\Gamma \left(\frac{1}{15}\right) \Gamma \left(\frac{2}{15}\right) \Gamma \left(\frac{4}{15}\right) \Gamma \left(\frac{8}{15}\right)}{240 \sqrt{5}}\label{eq:Bologna}
\end{align}in terms of Euler's gamma function  $ \Gamma(x):=\int_0^\infty u^{x-1}e^{-u}\D u$ for $x>0$. It is worth noting that nearly a decade had elapsed between the original proposal  \cite{Bailey2008,Laporta2008}  of \eqref{eq:Bologna} and its first rigorous (and highly technical) verification \cite{BlochKerrVanhove2015,Samart2016}.
Our
simplified proof of \eqref{eq:Bologna} draws on its connection to a ``random walk integral'' $ \JYM(5,0;1)$.

In \S\ref{sec:5Bessel}, we push the evaluation of  \eqref{eq:Bologna} one step further, to give explicit verifications of all the entries in the following $2\times 2$ matrix:\begin{align}
\begin{pmatrix}\IKM(1,4;1) & \IKM(1,4;3) \\
\IKM(2,3;1) & \IKM(2,3;3) \\
\end{pmatrix}=\begin{pmatrix}\pi^{2}C & \pi^2\left( \frac{2}{15} \right)^2\left( 13C-\frac{1}{10C} \right)\\
\frac{\sqrt{15}\pi}{2}C & \frac{\sqrt{15}\pi}{2}\left( \frac{2}{15} \right)^2\left( 13C+\frac{1}{10C} \right) \\
\end{pmatrix},\label{eq:BM_m2}
\end{align}where $ C=\frac{1}{240 \sqrt{5}\pi^{2}}\Gamma \left(\frac{1}{15}\right) \Gamma \left(\frac{2}{15}\right) \Gamma \left(\frac{4}{15}\right) \Gamma \left(\frac{8}{15}\right)$ is the ``Bologna constant'' attributed to Broadhurst \cite{Broadhurst2007,Bailey2008} and Laporta \cite{Laporta2008}. (Here, the rigorous evaluation of the top-right entry $ \IKM(1,4;3)$  was previously unattested in the literature.) We accomplish this by using a modular function of level 6 (\S\ref{subsec:RW_mod_form}) that parametrizes a Picard--Fuchs differential equation of third order (\S\ref{eq:Sym2}) attached to a family of $K3$ surfaces  formerly studied  by Bloch--Kerr--Vanhove \cite{BlochKerrVanhove2015} and Samart \cite{Samart2016}.
In addition to proving  \eqref{eq:BM_m2} in \S\ref{subsec:BM_det_5Bessel},
we work out  the Eichler integral representations of  $ \IKM(1,4;1) $, $\IKM(1,4;3)$ and $\IKM(1,4;5)$, which  involve contour integrals over certain holomorphic modular forms.

We devote \S\ref{sec:6Bessel} to the verification of the following integral formulae \cite[conjectured in][(109)--(111)]{Broadhurst2016}:{\allowdisplaybreaks\begin{align}\frac{3}{\pi^{2}}\IKM(1,5;1)=\IKM(3,3;1)={}&-6\pi^2\int_{0}^{i\infty}f_{4,6}(z)z\D z=\frac{3}{2}L(f_{4,6},2),\label{eq:IKM151_331_eta_int}\\\IKM(2,4;1)={}&\frac{\pi^{3}}{i}\int_{0}^{i\infty}f_{4,6}(z)\D z=\frac{\pi ^{2}}{2}L(f_{4,6},1)\notag\\={}&6\pi^{3}i\int_{0}^{i\infty}f_{4,6}(z)z^{2}\D z=\frac{3}{2}L(f_{4,6},3),\label{eq:IKM241_eta_int}
\end{align}}where \begin{align}
f_{4,6}(z)=[\eta(z)\eta(2z)\eta(3z)\eta(6z)]^{2}
\end{align}is a weight-4 modular form defined through the Dedekind eta function\begin{align} \eta(z):=e^{\pi iz/12}\prod_{n=1}^\infty(1-e^{2\pi inz}),\quad z\in\mathfrak H:=\{w\in\mathbb C|\I w>0\}.\label{eq:eta_defn}\end{align}
To prove  these formulae relating Bessel moments to critical $L$-values (a special $L$-value $ L(f,s) $ is said to be critical if $s$ is a positive integer less than the weight of the modular form $f$),  we use modular parametrizations of Hankel transforms and the Parseval--Plancherel identity.

In \S\ref{sec:8Bessel}, we fully exploit the techniques developed in the previous two sections, and confirm the following identities  \cite[cf.][(143)--(146)]{Broadhurst2016}:
\begin{align}\IKM(4,4;1)={}&4\pi^3 i\int_{0}^{i\infty}f_{6,6}(z)z^{2}\D z=L(f_{6,6},3),\\\frac{1}{\pi^{2}}\IKM(1,7;1)=\IKM(3,5;1)={}&6\pi^{4}\int_0^{i\infty} f_{6,6}(z)z^{3}\D z=\frac{9}{4}L(f_{6,6},4),\label{eq:IKM171_351_eta_int}\\\IKM(2,6;1)={}&\frac{9\pi^{5}}{i}\int_{0}^{i\infty}f_{6,6}(z)z^{4}\D z=\frac{27}{4}L(f_{6,6},5),\label{eq:IKM261_eta_int}\end{align}
which involve a weight-6 modular form \begin{align}f_{6,6}(z)=
\frac{ [\eta (2 z) \eta (3 z)]^{9}}{[\eta (z)\eta (6 z)]^{3}}+\frac{ [\eta ( z) \eta (6 z)]^{9}}{[\eta (2z)\eta (3 z)]^{3}}.
\end{align} In addition, we also use explicit computations to verify the Eichler--Shimura--Manin relation  $L(f_{6,6},5)/L(f_{6,6},3) =2\pi^{2}/21$ \cite[cf.][(142)]{Broadhurst2016}
and the sum rule     $ 9\pi^2\IKM(4,4;1)-14\IKM(2,6;1)=0$   \cite[cf.][(147)]{Broadhurst2016}.

Broadhurst has recently proposed a vast set of conjectures \cite{BroadhurstMellit2016,Broadhurst2016,Broadhurst2017Paris,Broadhurst2017CIRM,Broadhurst2017Higgs,Broadhurst2017DESY,Broadhurst2017ESIa,Broadhurst2017ESI} connecting Feynman diagrams to special values of Hasse--Weil $L$-functions, whose local factors arise from Kloosterman sums \cite[][\S\S2--6]{Broadhurst2016}. Our current work only  touches upon  $ \IKM(a,b;1)$ for $ a+b\in\{5,6,8\}$, where the corresponding $L$-series are modular. It is our hope that,  by verifying  a small subset of Broadhurst's thought-inspiring conjectures about Bessel moments, we could make first steps towards  an arithmetic understanding of these important mathematical constants deeply embedded in fundamental laws of nature, \textit{viz.}~quantum electrodynamics. On one hand, we have Feynman diagrams realized as motivic integrals, whose cohomology belongs to the realm of  algebraic geometry; on the other hand, these Feynman integrals also evaluate to arithmetic objects, such as  Eichler integrals and special $L$-values, whose symmetries embellish modern number theory.

\section{Bessel functions and their Wick rotations\label{sec:Hankel}}
\subsection{Some analytic properties of Bessel functions\label{subsec:HankelH1H2}}For $ \nu\in\mathbb C,-\pi<\arg z<\pi$,  the Bessel functions $ J_\nu$ and $Y_\nu$ are defined by\begin{align}
J_\nu(z):={}&\sum_{k=0}^\infty\frac{(-1)^k}{k!\Gamma(k+\nu+1)}\left( \frac{z}{2} \right)^{2k+\nu},&Y_\nu(z):={}&\lim_{\mu\to\nu}\frac{J_\mu(z)\cos(\mu\pi )-J_{-\mu}(z)}{\sin(\mu\pi)},\label{eq:JY_series}\intertext{which may be compared to the modified Bessel functions
$I_\nu$ and $K_\nu$:}I_\nu(z):={}&\sum_{k=0}^\infty\frac{1}{k!\Gamma(k+\nu+1)}\left( \frac{z}{2} \right)^{2k+\nu},&K_\nu(z):={}&\frac{\pi}{2}\lim_{\mu\to\nu}\frac{I_{-\mu}(z)-I_\mu(z)}{\sin(\mu\pi)}.\label{eq:IK_series}\end{align}Hereafter, the fractional powers of complex numbers are defined through $ w^\beta=\exp(\beta\log w)$ for $\log w=\log|w|+i\arg w$, where $ |\arg w|<\pi$.

We will also need the  cylindrical Hankel functions $H^{(1)}_0(z)=J_0(z)+iY_0(z)$ and  $H^{(2)}_0(z)=J_0(z)-iY_0(z)$  of zeroth order, which are both  well defined for  $ -\pi<\arg z<\pi$. In view of \eqref{eq:JY_series} and \eqref{eq:IK_series}, we can verify  \begin{align} J_0(ix)=I_0(x)\quad\text{ and}\quad\frac{\pi i}{2}H_0^{(1)}(ix)=K_0(x) \end{align} as well as \begin{align}
H_0^{(1)}(x+i0^+)=J_{0}(x)+i Y_0(x)\quad\text{and}\quad H_0^{(1)}(-x+i0^+)=-J_{0}(x)+i Y_0(x) \label{eq:H1_x}
\end{align}  for $ x>0$.

As $ |z|\to\infty,-\pi<\arg z<\pi$, we have the following asymptotic behavior: \begin{align}
H_0^{(1)}(z)=\sqrt{\frac{2}{\pi z}}e^{i\left(z-\frac{\pi}{4}\right)}\left[1+O\left( \frac{1}{|z|} \right)\right]\quad \text{and}\quad
H_0^{(2)}(z)=\sqrt{\frac{2}{\pi z}}e^{-i\left(z-\frac{\pi}{4}\right)}\left[1+O\left( \frac{1}{|z|} \right)\right].
\end{align}The asymptotic behavior of $ J_0(z)=[H_0^{(1)}(z)+H_0^{(2)}(z)]/2$ can be inferred accordingly.

\subsection{Contour deformations for Bessel moments\label{subsec:Wick}}

In the next lemma, we present a mechanism that generates cancelation formulae for $ \JYM$. Special cases of this lemma (involving four Bessel factors) have already appeared in \cite[][\S2]{Zhou2017Int4Pnu}.
\begin{lemma}[Bessel--Hankel--Jordan]\label{lm:BHJ}For $ \ell,m,n\in\mathbb Z_{\geq0}$ satisfying either $\ell-(m+n)/2<0; m<n$ or $ \ell-m=\ell-n<-1$, we have \begin{align}
\int_{i0^+-\infty}^{i0^++\infty}[J_0^{\vphantom{(1)}}(z)]^m[H_0^{(1)}(z)]^nz^\ell\D z:=\lim_{\varepsilon\to0^+}\lim_{R\to\infty}\int_{i\varepsilon-R}^{i\varepsilon+R}[J_0^{\vphantom{(1)}}(z)]^m[H_0^{(1)}(z)]^nz^\ell\D z=0.\label{eq:BHJ}
\end{align}\end{lemma}\begin{proof}As the integrand goes asymptotically like $ O(z^{\ell-(m+n)/2}e^{i(n-m)z})$ for $ \I z>0,|z|\to\infty$, we can  close the contour in the upper half-plane with the help of Jordan's lemma.\end{proof}\begin{remark}Noting \eqref{eq:H1_x} and $J_0(-x)=J_0(x) $, we may reformulate \eqref{eq:BHJ} as\begin{align}
\int_0^\infty [J_{0}(x)]^m\left\{ [J_{0}(x)+i Y_0(x)]^{n}+(-1)^\ell [-J_{0}(x)+i Y_0(x)]^{n}\right\}x^\ell\D x=0,\label{eq:JY_BHJ}\tag{\ref{eq:BHJ}$'$}
\end{align}which is a more convenient form to be used later.\eor\end{remark}

In addition to  closing the contour upwards (Lemma \ref{lm:BHJ}), sometimes we also need to turn the contour 90$^\circ$  clockwise, from the positive imaginary axis to the positive real axis. This trick is known as Wick rotation in QFT.
Instead of stating and justifying the general procedures for Wick rotations, we illustrate with a concrete example that relates $ \IKM(1,4;1)$ to a well-studied integral in probability theory.
\begin{theorem}[``Tiny nome of Bologna'']\label{thm:Bologna}We have \begin{align}
\int_{0}^\infty I_0(t)[K_0(t)]^4t\D t={}&\frac{\pi^{4}}{30}\int_0^\infty [J_0(x)]^5 x\D x=\frac{\Gamma \left(\frac{1}{15}\right) \Gamma \left(\frac{2}{15}\right) \Gamma \left(\frac{4}{15}\right) \Gamma \left(\frac{8}{15}\right)}{240 \sqrt{5}}.\label{eq:Bologna'}\tag{\ref{eq:Bologna}$'$}
\end{align}\end{theorem}\begin{proof}Thanks to  Jordan's lemma, we can deform the contour  in\begin{align}\left(\frac2\pi\right)^4
\int_{0}^\infty I_0(t)[K_0(t)]^4t\D t=-\R\int_{0}^{i\infty} J_0^{\vphantom{(1)}}(z)[H_0^{(1)}(z)]^4 z\D z,
\end{align}and identify it with its ``Wick-rotated'' counterpart:\begin{align}
-\R\int_{0}^{\infty} J_0^{\vphantom{(1)}}(x)[H_0^{(1)}(x)]^4 x\D x=-\int_{0}^\infty J(J^4-6 J^2 Y^2+ Y^4)x\D x,
\end{align}where $ J$ (resp.~$Y$) stands for $J_0(x)$ (resp.~$ Y_0(x)$) in the last expression. Now that \begin{align}
J(J^4-6 J^2 Y^2+ Y^4)-\frac{2J^2}{3}  [(J+i Y)^3-(-J+i Y)^3]-\frac{(J+i Y)^5-(-J+i Y)^5}{10} =-\frac{8 J^5}{15} ,
\end{align}we can verify the first equality in \eqref{eq:Bologna'}, while referring back to \eqref{eq:JY_BHJ} in Lemma \ref{lm:BHJ}.

The ``random walk integral'' $ \int_0^\infty [J_0(x)]^5 x\D x$ has been thoroughly studied by  Borwein and coworkers \cite{BSWZ2012}. One can evaluate this integral through a special value of a modular form (to be elaborated later in \S\ref{subsec:RW_mod_form}). Here, we simply point out that the second equality in \eqref{eq:Bologna'} can be directly deduced from \cite[][(5.2)]{BSWZ2012}.
 \end{proof}
\begin{remark}We pause to give a brief account for the history of the integral identity in  \eqref{eq:Bologna}. The closed-form evaluation in \eqref{eq:Bologna} was initially proposed by Broadhurst in the form of elliptic theta functions \cite[][(93)]{Bailey2008}, and the current (equivalent) form involving products of gamma functions was suggested by  Laporta \cite[][(7), (16), (17)]{Laporta2008}.
Bloch--Kerr--Vanhove studied the momentum space reformulation of  $ \IKM(1,4;1)$ as a triple integral of an algebraic function over the first octant:\begin{align}
\IKM(1,4;1)=\frac{1}{8}\int_0^\infty\frac{\D X}{X}\int_0^\infty\frac{\D Y}{Y}\int_0^\infty\frac{\D Z}{Z}\frac{1}{(1+X+Y+Z)(1+X^{-1}+Y^{-1}+Z^{-1})-1},
\end{align} with a \textit{tour de force}      in motivic cohomology. They effectively verified   \eqref{eq:Bologna} by casting $ \IKM(1,4;1)$ into $ \frac{\pi ^3}{8 \sqrt{15}}\frac{[\eta(z)\eta(3z)]^{4}}{[\eta(2z)\eta(6z)]^{2}}$ for $ z=\frac{3+i\sqrt{15}}{6}$ \cite[][(2.5.9)]{BlochKerrVanhove2015}.  Drawing on a result of Rogers--Wan--Zucker \cite[][Theorem 5]{RogersWanZucker2015}, Samart reanalyzed the triple integral formulation of   $ \IKM(1,4;1)$, before finally expressing $ \IKM(1,4;1)$ as explicit gamma factors, and identifying it with a special $L$-value $\frac{3\pi}{2\sqrt{15}}L(f_{3,15},2) $  for the modular form $ f_{3,15}(z)=[\eta(3z)\eta(5z)]^3+[\eta(z)\eta(15z)]^3$ \cite[(35)]{Samart2016}.   \eor\end{remark}
\begin{remark}In \cite[][\S5]{BSWZ2012}, the authors remarked on the uncanny resemblance of the ``random walk integral''  $ \int_0^\infty [J_0(x)]^5 x\D x$ to the ``tiny nome of Bologna'', without supplying a mechanistic interpretation later afterwards. Moreover, these authors recorded \cite[][Remark 7.3]{BSWZ2012}\begin{align}\frac{4}{\pi^3}\int_0^\infty [K_0(t)]^3\D t=\int_0^\infty[J_0(x)]^3\D x\end{align} and  \cite[][between Theorems 7.6 and 7.7]{BSWZ2012}\begin{align}
\frac{4}{\pi^3}\int_0^\infty I_0(t)[K_0(t)]^3\D t=\int_0^\infty[J_0(x)]^4\D x
\end{align}after comparing explicit expressions of all the integrals in question, probably unaware that such equalities would follow easily from a Wick rotation and an application of Lemma \ref{lm:BHJ} above.
   \eor\end{remark}

\section{Feynman diagrams with 5 Bessel factors\label{sec:5Bessel}}
\subsection{A modular form associated with Bessel moments\label{subsec:RW_mod_form}}
In this paper, we will mainly deal with    modular forms of  level 6, which respect the symmetries in the Hecke congruence group\begin{align}
\varGamma_0(6):=\left\{ \left.\begin{pmatrix}a & b \\
c & d \\
\end{pmatrix}\right|a,b,c,d\in\mathbb Z,ad-bc=1,c \equiv0\,(\!\bmod6)\right\}.
\end{align}Furthermore, following the notation of Chan--Zudilin \cite{ChanZudilin2010}, we write $ \widehat W_3=\frac{1}{\sqrt{3}}\left(\begin{smallmatrix}3&-2\\6&-3\end{smallmatrix}\right)$ and construct a group $ \varGamma_0(6)_{+3}=\langle\varGamma_0(6),\widehat W_3\rangle$  by adjoining $ \widehat W_3$ to $ \varGamma_0(6)$.
To set the stage for  later developments  in this article, we present some characteristics of a modular function on $ \varGamma_0(6)_{+3}$.\begin{lemma}[A modular function of level 6]\label{lm:X63}The function $ X_{6,3}(z):=\left[\frac{ \eta (2z ) \eta (6 z )}{\eta (z ) \eta (3z)}\right]^{6},z\in\mathfrak H$ has the following properties:
\begin{align}\left\{ \begin{array}{r@{\;=\;}l@{ \quad \text{if}\quad}r@{\;\in\;}l}
X_{6,3}\left( \tfrac{az+b}{cz+d} \right)&X_{6,3}(z),&\left(\begin{smallmatrix}a&b\\c&d\end{smallmatrix}\right)  &\varGamma_0(6)_{+3};\\
\I X_{6,3}(z)&0,&2\R z&\mathbb Z;\\
X_{6,3}\left( \tfrac12+\tfrac{iy}{2\sqrt{3}} \right)&X_{6,3}\left( \tfrac12+\tfrac{i}{2\sqrt{3}y} \right),&y&(0,\infty).
\end{array} \right.
\label{eq:X63_symm}
\end{align}
Moreover, the following mappings \begin{align}\left\{ \begin{array}{r@{\,:\,}r@{\;\longrightarrow\;}l}X_{6,3}&\{z|\R z=0,\I z>0\}&(0,\infty)\\X_{6,3}&\left\{z\left|\R z=\tfrac12,\I z>\tfrac{1}{2\sqrt{3}}\right.\right\}&\left(-\tfrac1{16},0\right)\end{array} \right.\label{eq:X63_para}
\end{align}are bijective.\end{lemma}\begin{proof}The function $ X_{6,3}$ is a Hauptmodul of $ \varGamma_0(6)_{+3}$ with genus 0 \cite[][(2.2)]{ChanZudilin2010}, so it must satisfy the modular invariance relation, as displayed  in the first line of  \eqref{eq:X63_symm}. To prove the second line in   \eqref{eq:X63_symm}, use the infinite product expansion for the Dedekind eta function in \eqref{eq:eta_defn}.
To prove the last line in \eqref{eq:X63_symm},  note that \begin{align}
\widehat W_3z=\frac{3z-2}{6z-3}=\frac12+\frac{i}{2\sqrt{3}y}\quad \text{for }z=\frac12+\frac{iy}{2\sqrt{3}}.
\end{align}

The domains of the mappings in \eqref{eq:X63_para} are proper subsets of the fundamental domain for  $ \varGamma_0(6)_{+3}$, so these mappings are necessarily injective. Furthermore, by the second line in   \eqref{eq:X63_symm}, these mappings are continuous real-valued functions defined on path-connected sets, so these   injective mappings must also be monotone along the respective paths, and their continuous images are also path-connected. Consequently, the modular function $ X_{6,3}$ induces bijective mappings from these two domains to their respective  ranges, and the extent of the latter is inferred from the ``boundary values'' of the function $ X_{6,3}$ at the extreme points of the domains of definition.       \end{proof}

As a demonstration for the relevance of modularity in our studies of Bessel moments, we recall some known results from \cite{Rogers2009,BSWZ2012}, in slightly reorganized form.
In particular, we will use the Chan--Zudilin notation    $ Z_{6,3}(z)=\frac{[\eta(z)\eta(3z)]^{4}}{[\eta(2z)\eta(6z)]^{2}}$     \cite[][(2.5)]{ChanZudilin2010} for a  modular form of weight 2 on $\varGamma_0(6)_{+3} $.

\begin{proposition}[Bessel moments as modular forms]\label{prop:BM_mod}For $ z/i>0$, we have\begin{align}
\int_0^\infty J_{0}\left(\left[\frac{2 \eta (2z ) \eta (6 z )}{\eta (z ) \eta (3z)}\right]^{3}t\right)I_0(t)[K_0(t)]^3t\D t={}&\frac{\pi ^2}{16}Z_{6,3}(z),\label{eq:IKKK_Hankel_eta}
\intertext{which gives a modular parametrization of $ \int_0^\infty J_{0}(xt)I_0(t)[K_0(t)]^3t\D t$ for $x>0 $. For $ z=\frac12+iy,y>\frac{1}{2\sqrt{3}}$, we have}\int_0^\infty I_{0}\left(\frac{1}{i}\left[\frac{2 \eta (2z ) \eta (6 z )}{\eta (z ) \eta (3z)}\right]^{3}t\right)I_0(t)[K_0(t)]^3t\D t={}&\frac{\pi ^2}{16}Z_{6,3}(z) \label{eq:IKKK_I}\intertext{and}
\int_0^\infty J_{0}\left(\frac{1}{i}\left[\frac{2 \eta (2z ) \eta (6 z )}{\eta (z ) \eta (3z)}\right]^{3}t\right)[J_0(t)]^4t\D t={}&\frac{3(2z-1)}{4\pi i}Z_{6,3}(z),\label{eq:JJJJ_Hankel_eta}
\end{align}which give modular parametrizations of  $ \int_0^\infty I_{0}(xt)I_0(t)[K_0(t)]^3t\D t$  and $ \int_0^\infty J_{0}(xt)[J_0(t)]^4t\D t$ for $x\in(0,2) $.
\end{proposition}\begin{proof}We   recall   from  \cite[][(55) and (56)]{Bailey2008} the following formula\begin{align}
\int_0^\infty I_0(t)[K_0(t)]^3t^{2n+1}\D t=\frac{\pi^{2}}{16}\left( \frac{n!}{4^{n}} \right)^2\sum^{n}_{k=0}{n\choose k}^2{2(n-k)\choose n-k}{2k\choose k}=\frac{\pi^{2}}{16}\left( \frac{n!}{4^{n}} \right)^2D_{n},\label{eq:IKKK_Domb}
\end{align}where   ${m\choose j}=\frac{m!}{j!(m-j)!}$ and $D_n$ is the $n$th Domb number. Meanwhile, we note that Rogers has shown   in \cite[][Theorem 3.1]{Rogers2009} that  \begin{align}
_3F_2\left(\left.\begin{array}{c}
\frac{1}{3},\frac{1}{2},\frac{2}{3} \\1,1 \\
\end{array}\right|\frac{27u^{2}}{4(1-u)^{3}}\right)=(1-u)\sum_{n=0}^\infty \frac{D_n}{4^n}u^n\label{eq:Rogers2009}
\end{align}holds for $|u|$ sufficiently small, where  \begin{align}{_pF_q}\left(\left.\begin{array}{c}
a_{1},\dots,a_p \\[4pt]
b_{1},\dots,b_q \\
\end{array}\right| x\right):=1+\sum_{n=1}^\infty\frac{\prod_{j=1}^p\frac{\Gamma(a_{j}+n)}{\Gamma(a_{j})}}{\prod_{k=1}^q\frac{\Gamma(b_k+n)}{\Gamma(b_k)} }\frac{x^n}{n!}.\label{eq:defn_pFq}\end{align} By termwise summation, we see that \begin{align}
\int_0^\infty J_{0}(xt)I_0(t)[K_0(t)]^3t\D t=\frac{\pi^{2}}{16+x^{2}}
{_3F_2}\left(\left.\begin{array}{c}
\frac{1}{3},\frac{1}{2},\frac{2}{3} \\1,1 \\
\end{array}\right|\frac{108x^{4}}{(16+x^{2})^{3}}\right)
\end{align}is valid for $x$ sufficiently small.    Parametrizing the right-hand side of the equation above with modular forms (see \cite[][(2.8)]{ChanZudilin2010} or  \cite[][(4.13)]{BSWZ2012}), we observe that \eqref{eq:IKKK_Hankel_eta} holds when $\I z$ is  sufficiently large and positive. By analytic continuation, the validity of  \eqref{eq:IKKK_Hankel_eta}  extends to the entire positive $\I z$-axis, from which $ x=\left[\frac{2 \eta (2z ) \eta (6 z )}{\eta (z ) \eta (3z)}\right]^{3}$ maps bijectively to $x\in(0,\infty)$.

Performing further analytic continuation on   \eqref{eq:IKKK_Hankel_eta}, we arrive at \eqref{eq:IKKK_I}.   Here, according to Lemma \ref{lm:X63}, we know that $ x=\frac{1}{i}\left[\frac{2 \eta (2z ) \eta (6 z )}{\eta (z ) \eta (3z)}\right]^{3}$ maps  $ y\in\left(\frac{1}{2\sqrt{3}},\infty\right)$ bijectively to $ x\in(0,2)$.

The integral identity in  \eqref{eq:JJJJ_Hankel_eta}  paraphrases \cite[][(4.16)]{BSWZ2012}. (A special case of this modular pa\-ra\-me\-tri\-za\-tion led to a closed-form evaluation of the  ``random walk integral'' $ \int_0^\infty [J_0(x)]^5 x\D x$ in \cite[][(5.2)]{BSWZ2012}, which we quoted in our proof of Theorem \ref{thm:Bologna}. See also Table~\ref{tab:spec_X63_Z63}.)  \end{proof}\begin{table}[t]\scriptsize\caption{Values of  $X_{6,3}(z)$,  $Z_{6,3}(z)$ and their derivatives at $z=\frac{1}{2}+\frac{i\sqrt5}{2\sqrt{3}}$ \big[with ``rescaled Bologna constant'' $ c=\sqrt{5}C=\frac{1}{240\pi^{2}}\Gamma \left(\frac{1}{15}\right) \Gamma \left(\frac{2}{15}\right) \Gamma \left(\frac{4}{15}\right) \Gamma \left(\frac{8}{15}\right)$\big] \label{tab:spec_X63_Z63}}\begin{minipage}{0.35\textwidth}
\begin{align*}\begin{array}{c|c}\hline\hline\vphantom{\dfrac{\frac11}{\frac11}}  \vphantom{\dfrac{\frac11}{\frac11}}X_{6,3}(z)&-\dfrac{1}{64}\\\vphantom{\dfrac{\frac11}{\frac11}}X_{6,3}'(z)&\dfrac{3\sqrt{15}c}{32i} \\\vphantom{\dfrac{\frac11}{\frac11}}X_{6,3}''(z)&\dfrac{9c (9 c+1)}{16} \\\vphantom{\dfrac{\frac11}{\frac11}}X_{6,3}'''(z)&\dfrac{27\sqrt{15}c (18 c^2-18 c-1)}{80i} \\\vphantom{\dfrac{\frac11}{\frac11}}X_{6,3}''''(z)&\dfrac{81c(753 c^3+54 c^2-27 c-1)}{20}  \\\hline\hline\end{array}\end{align*}\end{minipage}\begin{minipage}{0.45\textwidth}
\begin{align*}\begin{array}{c|c}\hline\hline\vphantom{\dfrac{\frac11}{\frac11}}  Z_{6,3}(z)&\vphantom{\dfrac{\frac12}{}}\dfrac{8\sqrt{3}c}{\pi}\\\vphantom{\dfrac{\frac11}{\frac11}}Z_{6,3}'(z)&\dfrac{48 c (3c-1)}{\sqrt{5}\pi  i}\\\vphantom{\dfrac{\frac11}{\frac11}}Z_{6,3}''(z)&-\dfrac{48 \sqrt{3} c (62 c^2-18 c+3)}{5 \pi }\\\vphantom{\dfrac{\frac11}{\frac11}}Z_{6,3}'''(z)&\dfrac{1728 i c (57 c^3-62 c^2+9 c-1)}{5 \sqrt{5} \pi }\\\vphantom{\dfrac{\frac11}{\frac11}}Z_{6,3}''''(z)&\dfrac{1728 \sqrt{3} c (266 c^4-228 c^3+124 c^2-12 c+1)}{5 \pi } \\\hline\hline\end{array}\end{align*}\end{minipage}
\end{table}\begin{remark}

For any CM point $ z\in\mathfrak H$ (a complex number in the upper half-plane that solves a quadratic equation with integer coefficients), the  absolute value $ |\eta(z)|$ of the Dedekind eta function $\eta(z)$ can be explicitly written as
the product of an algebraic number, a rational power of $\pi$, and rational powers of special values for Euler's gamma function (see \cite[][\S12]{SelbergChowla} or \cite[][Theorem 9.3]{PoortenWilliams1999}). At any CM\ point $z$, the following expressions are computable algebraic numbers \cite[][(1.2.9) and Appendix 1]{AGF_PartI}: \begin{align}
\frac{E_2(z)}{[\eta(z)]^{4}}=\frac{12}{\pi i[\eta(z)]^{4}}\left[ \frac{\D\log\eta(z)}{\D z} -\frac{i}{4\I z}\right],\quad \frac{E_4(z)}{[\eta(z)]^8},\quad \frac{E_6(z)}{[\eta(z)]^{12}},
\end{align} where \begin{align} E_4(z)=1+240\sum_{n=1}^\infty \frac{n^3e^{2\pi inz}}{1-e^{2\pi i nz}},\quad E_6(z)=1-504\sum_{n=1}^\infty \frac{n^5e^{2\pi inz}}{1-e^{2\pi i nz}}\end{align} are  Eisenstein series of weights 4 and 6.
 Higher order derivatives of the Dedekind eta function can be deduced from Ramanujan's differential equations \cite{Ramanujan1916}:\begin{align}\left\{\begin{array}{r@{\;=\;}l}\dfrac{1}{2\pi i}\dfrac{\D E^*_2(z)}{\D z}&\dfrac{[E^*_2(z)]^2-E_4(z)}{12},\\[8pt]\dfrac{1}{2\pi i}\dfrac{\D E_4(z)}{\D z}&\dfrac{E^*_2(z)E_4(z)-E_6(z)}{3},\\[8pt]\dfrac{1}{2\pi i}\dfrac{\D E_6(z)}{\D z}&\dfrac{E^*_2(z)E_6(z)-[E_4(z)]^{2}}{2},\end{array}\right.\label{eq:PQR_ODE}
\end{align} where $ E_2^*(z)=E_2(z)+\frac{3}{\pi\I z}$ is a holomorphic ``weight-2 Eisenstein series''.

Samart has computed the values of $X_{6,3}(z) $ and $ Z_{6,3}(z)$ at $ z=\frac{1}{2}+\frac{i\sqrt{5}}{2\sqrt{3}}$ explicitly  \cite[][Lemma 1]{Samart2016}. We may combine his results with  \eqref{eq:PQR_ODE} to evaluate derivatives of  $X_{6,3}(z) $ and $ Z_{6,3}(z)$ at $ z=\frac{1}{2}+\frac{i\sqrt{5}}{2\sqrt{3}}$, as summarized in  Table~\ref{tab:spec_X63_Z63}.     \eor\end{remark}\begin{remark}As the Bessel differential equation leaves us \cite[][\S1]{Bailey2008}\begin{align}
\left(\frac{\partial^{2}}{\partial x^{2}}+\frac{1}{x}\frac{\partial}{\partial x}\right)^kI_0(xt)=t^{2k} I_0(xt),\quad \forall k\in\mathbb Z_{>0},\label{eq:Bessel_diff_trick}
\end{align}we will have no difficulties in computing $\IKM(2,3;1)=\frac{\sqrt{15}\pi}{2}C $, $ \IKM(2,3;3)=\frac{\sqrt{15}\pi}{2}\left( \frac{2}{15} \right)^2\big( 13C+\frac{1}{10C} \big)$ and $ \IKM(2,3;5)=\frac{\sqrt{15}\pi}{2}\left( \frac{4}{15} \right)^3\big( 43C+\frac{19}{40C} \big)$ from \eqref{eq:IKKK_I}, with assistance from Table~\ref{tab:spec_X63_Z63}. These Bessel moments were previously evaluated  in \cite[][\S5.10]{Bailey2008} with  combinatorial techniques.\eor\end{remark}
\subsection{Symmetric squares and Eichler integrals\label{eq:Sym2}}Central to the studies of Bloch--Kerr--Vanhove \cite{BlochKerrVanhove2015} and Samart \cite{Samart2016}  was the following motivic integral: \begin{align}\mathscr I(u):={}&\int_0^\infty I_0(\sqrt{u}t)[K_0(t)]^{4}t\D t\notag
\\={}&\frac{1}{8}\int_0^\infty\frac{\D X}{X}\int_0^\infty\frac{\D Y}{Y}\int_0^\infty\frac{\D Z}{Z}\frac{1}{(1+X+Y+Z)(1+X^{-1}+Y^{-1}+Z^{-1})-u},\label{eq:XYZ_int}
\end{align} and the geometry for  the family of $K3$ surfaces that compactify the locus  of $ (1+X+Y+Z)(1+X^{-1}+Y^{-1}+Z^{-1})-u=0$ and resolve singularities. Inspired by their analysis, we give a modular parametrization of $ \mathscr I(u)$  for $ u\leq16$.
In \cite{BlochKerrVanhove2015} and  \cite{Samart2016}, the authors  parametrized  the Feynman integral $ \mathscr I(u)$  with the modular function  $u(z)=-\left[\frac{\eta(z)\eta(3z)}{\eta(2z)\eta(6z)}\right]^6$, and needed sophisticated computations at the CM\ point $ z_*=\frac{-3+i\sqrt{15}}{24}$ where $u(z_*)=1$. In what follows, we will use a different modular parametrization (Lemma \ref{lm:Jac_xz}) to facilitate the representation of Bessel moments via Eichler integrals (Proposition \ref{prop:Eichler_Iu}).

\begin{lemma}[Jacobian for a modular function]\label{lm:Jac_xz}The modular parametrization   \begin{align} x=\frac{1}{i}\left[\frac{2 \eta (2z ) \eta (6 z )}{\eta (z ) \eta (3z)}\right]^{3}\end{align}  satisfies\begin{align}
\frac{1}{x}\frac{\D x}{\D z}=\pi i\left\{ \frac{[\eta (z ) \eta (2z )]^3}{\eta (3z ) \eta (6z )}+9\frac{ [\eta (3z ) \eta (6z )]^3}{\eta (z ) \eta (2z )} \right\}.\label{eq:xz_Jac}
\end{align} With $ q=e^{2\pi iz}$,  we have the following asymptotic behavior\begin{align}
q\frac{\D x}{\D q}=\frac{1}{2\pi i}\frac{\D x}{\D z}=\frac{4\sqrt{q}}{i}[1 + 9 q + 30 q^2 + 112 q^3 + 297 q^4+O(q^5)]\label{eq:xz_Jac_near_cusp}
\end{align} near the infinite cusp  ($z\to i \infty, q\to0$).\end{lemma}\begin{proof}We can verify the following identity \begin{align}
\frac{\D}{\D z}\log\frac{\eta (2z )}{\eta (z ) }=\frac{\pi i}{12}\frac{[\eta(z)]^{8}+32[\eta(4z)]^{8}}{[\eta(2z)]^{4}},\quad\forall z\in\mathfrak H\label{eq:logWeberf2_deriv}
\end{align} by showing that the ratio between both sides defines a bounded function on the compact Riemann surface $ X_{0}(2)=\varGamma_0(2)\backslash(\mathfrak H\cup\mathbb Q\cup\{i\infty\})$, and that this ratio  tends to $1$ as $z$ approaches the infinite cusp. Employing an identity due to Chan--Zudilin \cite[][(4.3)]{ChanZudilin2010}, we  rewrite  \eqref{eq:logWeberf2_deriv} as\begin{align}\frac{\D}{\D z}\log\frac{\eta (2z )}{\eta (z ) }={}&\frac{\pi i}{12}\left\{ \frac{[\eta (z ) \eta (2z )]^3}{\eta (3z ) \eta (6z )}+27\frac{ [\eta (3z ) \eta (6z )]^3}{\eta (z ) \eta (2z )} \right\},&\forall z\in\mathfrak H.\label{eq:logWeberf2_deriv'}
\tag{\ref{eq:logWeberf2_deriv}$'$}
\intertext{Meanwhile, a cubic transformation brings us \cite[][second equation below (4.5)]{ChanZudilin2010}}
\frac{\D}{\D z}\log\frac{\eta (6z )}{\eta (3z ) }={}&\frac{\pi i}{4}\left\{ \frac{[\eta (z ) \eta (2z )]^3}{\eta (3z ) \eta (6z )}+3\frac{ [\eta (3 z ) \eta (6z )]^3}{\eta (z ) \eta (2z )} \right\},&\forall z\in\mathfrak H.\label{eq:logWeberf2_deriv''}
\tag{\ref{eq:logWeberf2_deriv}$''$}
\end{align}The two equations above add up to \eqref{eq:xz_Jac}.

The expansion in \eqref{eq:xz_Jac_near_cusp} follows directly from  \eqref{eq:xz_Jac} and $ \eta(z)=q^{1/24}\prod_{n=1}^\infty(1-q^n)$. \end{proof}\begin{proposition}[Eichler integral representation of $ \mathscr I(u)$]\label{prop:Eichler_Iu}
Let $ \zeta(3)=\sum_{n=1}^\infty n^{-3}$ be Ap\'ery's constant. For $ z/i>0$, we have\begin{align}&
\int_0^\infty J_{0}\left(\left[\frac{2 \eta (2z ) \eta (6 z )}{\eta (z ) \eta (3z)}\right]^{3}t\right)[K_0(t)]^4t\D t\notag\\={}&Z_{6,3}(z)\left[\frac{7\zeta(3)}{8}+12\pi ^{3}i\int^{i\infty}_z\left[\frac{\eta (2 z') \eta (6 z')}{\eta (z') \eta (3 z')}\right]^3 \left\{[\eta (z') \eta (2 z')]^4+9[ \eta (3 z')\eta (6 z')]^4\right\}(z-z')^2\D z'\right],\label{eq:KKKK_Hankel_eta}
\intertext{which parametrizes $ \int_0^\infty J_{0}(xt)[K_0(t)]^4t\D t$ for $x>0$. For $ z=\frac12+iy,y\in\bigg(\frac{1}{2\sqrt{3}},\infty\bigg)$, we have}&\int_0^\infty I_{0}\left(\frac{1}{i}\left[\frac{2 \eta (2z ) \eta (6 z )}{\eta (z ) \eta (3z)}\right]^{3}t\right)[K_0(t)]^4t\D t\notag\\={}&Z_{6,3}(z)\left[\frac{7\zeta(3)}{8}+12\pi ^{3}i\int^{i\infty}_z\left[\frac{\eta (2 z') \eta (6 z')}{\eta (z') \eta (3 z')}\right]^3 \left\{[\eta (z') \eta (2 z')]^4+9[ \eta (3 z')\eta (6 z')]^4\right\}(z-z')^2\D z'\right],\label{eq:IKKKK_Eichler}\end{align}which parametrizes  $ \int_0^\infty I_{0}(xt)[K_0(t)]^4t\D t$ for $x\in(0,2)$. Moreover, the equation above remains valid for  $ z=\frac12+iy,y\in\bigg(0,\frac{1}{2\sqrt{3}}\bigg)$, corresponding to  $x\in(0,2)$; and for    $ z=\frac12+\frac{i}{2\sqrt{3}}e^{i\varphi},\varphi\in[0,\pi/3]$, corresponding to $ x\in[2,4]$. \end{proposition}\begin{proof}Unlike the expressions    $ \int_0^\infty I_{0}(xt)I_0(t)[K_0(t)]^3xt\D t$  and $ \int_0^\infty J_{0}(xt)[J_0(t)]^4xt\D t$ (covered in Proposition \ref{prop:BM_mod}), which are annihilated by the Picard--Fuchs operator \cite[][(2.6) and (2.7)]{BSWZ2012}\begin{align}
\widehat A_4:={}&x^{4}\left( x\frac{\D}{\D x}+1 \right)^3-4x^{3}\frac{\D}{\D x}\left[5\left(x\frac{\D}{\D x}\right)^{2}+3 \right]+64\left( x\frac{\D}{\D x}-1 \right)^3\notag\\={}&(x - 4) (x - 2) x^3 (x + 2) (x + 4)\frac{\D^{3}}{\D x^{3}}+6 x^4 (x^2 - 10)\frac{\D^{2}}{\D x^{2}}\notag\\{}&+x (7 x^4 - 32 x^2 + 64)\frac{\D}{\D x}+(x^2 - 8) (x^2 + 8),\label{eq:op_A4}
\end{align}the function $ x\mathscr I(x^2)=\int_0^\infty I_{0}(xt)[K_0(t)]^4xt\D t$ satisfies an inhomogeneous differential equation \cite[cf.][Theorem 2.2.1]{BlochKerrVanhove2015}:\begin{align}
\widehat A_4[x\mathscr I(x^2)]=-24x^{3}.\label{eq:A4_inhom}
\end{align}

For a solution to the homogeneous equation $ \widehat A_4[f(x)]=0$, a modular parametrization  \cite[cf.][Remark 4.10]{BSWZ2012}  $ x=\frac{1}{i}\left[\frac{2 \eta (2z ) \eta (6 z )}{\eta (z ) \eta (3z)}\right]^{3}$ leaves us general solutions in the form of \begin{align}
\frac{f(x)}{x}=Z_{6,3}(z)(c_{0}+c_1z+c_2 z^2),\label{eq:fx_Sym2}
\end{align}where the constants $c_0$, $c_1$, $c_2$ can be determined by the  behavior of $f(x)$ in specific contexts. We have the simple functional form in \eqref{eq:fx_Sym2} because the operator $\widehat A_4$ is a symmetric square \cite[][Remark 4.6]{BSWZ2012} and the corresponding family of $K3$ surfaces $ (1+X+Y+Z)(1+X^{-1}+Y^{-1}+Z^{-1})=u$ admit Shioda--Inose structure (see \cite[][Corollary 7.1]{Morrison1984K3}, \cite[][\S3.2]{BlochKerrVanhove2015} and \cite[][\S5]{Samart2016}).

To construct a particular solution to the inhomogeneous equation in \eqref{eq:A4_inhom}, we follow the Bloch--Kerr--Vanhove recipe \cite[][(2.3.9)]{BlochKerrVanhove2015}, and derive the differential equation for the Wro\'nskian determinant $W(x) $ via\begin{align}
\frac{\D }{\D x}\log W(x)={}&-\frac{6 x^4 (x^2 - 10)}{(x - 4) (x - 2) x^3 (x + 2) (x + 4)}\notag\\={}&-\frac{3}{2}\frac{\D }{\D x}\log[(16-x^2)(4-x^2)].
\end{align}Here, we determine the normalizing constant $ \kappa=1024i/\pi^{3}$ for  the Wro\'nskian \begin{align} W(x)=\frac{\kappa}{[(16-x^2)(4-x^2)]^{3/2}}=\det\begin{pmatrix}y_{0}(x)     & y_1(x) & y_{2}(x) \\
y_{0}'(x) & y_{1}'(x) & y_{2}'(x) \\
y_{0}''(x) & y_{1}''(x) & y_{2}''(x) \\
\end{pmatrix}\end{align} by  choosing a basis\begin{align}
\frac{y_{j}(x)}{x}= Z_{6,3}(z)z^j,\quad j\in\{0,1,2\},
\end{align}differentiating in $x$ with the help of \eqref{eq:xz_Jac_near_cusp} in Lemma  \ref{lm:Jac_xz} for small values of $ q=e^{2\pi iz}\to0$, and extracting the leading coefficient in the $q$-expansion $ \frac{\kappa}{512}[1-30 q+474 q^2+O(q^{3})]=\frac{2 i}{\pi ^3}[1-30 q+474 q^2+O(q^{3})]$. Then, we  simplify the integral representation of a particular solution \cite[cf.][(2.3.8)]{BlochKerrVanhove2015}
\begin{align}
y_*(\mathscr X)=\int_0^{\mathscr X}\frac{\widetilde W(\mathscr X,x)\hat A_4[x\mathscr I(x^2)]\D x}{W(x)(x - 4) (x - 2) x^3 (x + 2) (x + 4)}
\end{align}
where\begin{align}
\widetilde W(\mathscr X,x)= \det\begin{pmatrix}y_{0}(x)     & y_1(x) & y_{2}(x) \\
y_{0}'(x) & y_{1}'(x) & y_{2}'(x) \\
y_{0}(\mathscr X) & y_{1}(\mathscr X) & y_{2}(\mathscr X) \\
\end{pmatrix},
\end{align}using the cofactors\begin{align}
\left\{\begin{array}{r@{\;=\;}l}\det\begin{pmatrix}y_{0}(x)     & y_1(x) \\
y_{0}'(x) & y_{1}'(x) \\
\end{pmatrix}&x^{2}[Z_{6,3}(z)]^{2}\dfrac{\D z}{\D x},\\\det \begin{pmatrix}y_{0}(x)     & y_2(x) \\
y_{0}'(x) & y_{2}'(x) \\
\end{pmatrix}&x^{2}[Z_{6,3}(z)]^{2}\dfrac{\D z}{\D x}(2z),\\\det\begin{pmatrix}y_{1}(x)     & y_2(x) \\
y_{1}'(x) & y_{2}'(x) \\
\end{pmatrix}&x^{2}[Z_{6,3}(z)]^{2}\dfrac{\D z}{\D x}z^{2}.\end{array} \right.
\end{align}With the parametrization   $ x=\frac{1}{i}\left[\frac{2 \eta (2z ) \eta (6 z )}{\eta (z ) \eta (3z)}\right]^{3},\mathscr X=\frac{1}{i}\left[\frac{2 \eta (2\mathscr Z) \eta (6\mathscr Z )}{\eta (\mathscr Z ) \eta (3\mathscr Z)}\right]^{3}$, we see that the general solution $f(\mathscr X) $ to the inhomogeneous equation $\widehat A_4 f(\mathscr X)=-24\mathscr X^3 $ is\begin{align}&
\mathscr X Z_{6,3}(\mathscr Z)(c_{0}+c_1\mathscr Z+c_2 \mathscr Z^2)+\notag\\&+12\pi^{3}i\mathscr X Z_{6,3}(\mathscr Z)\int^{i\infty}_{\mathscr Z}\sqrt{\smash[b]{1+4X_{6,3}(z)}}\sqrt{\smash[b]{1+16X_{6,3}(z)}}[Z_{6,3}(z)]^{2}X_{6,3}(z)(\mathscr Z-z)^2\D z.
\end{align}Since $ Z_{6,3}(z)\to1$ as $ z\to i\infty$, we must have \begin{align}
c_{0}=\IKM(0,4;1)=\int_0^\infty [K_0(t)]^4 t\D t=\frac{7\zeta(3)}{8},c_1=0,c_2=0
\end{align}for our Eichler integral representations of Bessel moments.

When  $ z/i>0$ or $ z=\frac{1}{2}+iy$ for $ y>\frac{1}{2\sqrt{3}}$, according to Chan--Zudilin \cite[][(3.3) and (3.5)]{ChanZudilin2010}, we have\begin{align}&
\sqrt{\smash[b]{1+4X_{6,3}(z)}}\sqrt{\smash[b]{1+16X_{6,3}(z)}}\notag\\={}&\frac{[\eta(2z)\eta(6z)]^2}{[\eta(z)\eta(3z)]^4}\left( \sum_{m,n\in\mathbb Z}e^{2\pi i(m^2+mn+n^2)z} \right)\left( \sum_{m,n\in\mathbb Z}e^{4\pi i(m^2+mn+n^2)z} \right),
\end{align}where the two double sums appear in Ramanujan's cubic theory for elliptic functions \cite[][Chap.~33]{RN5}. Meanwhile, Borwein--Borwein--Garvan \cite[][Proposition 2.2(i)(ii) and Theorem 2.6(i)]{BorweinBorweinGarvan} identified  the product of these two double sums with\begin{align}
\frac{[\eta(z)\eta(2z)]^{3}}{\eta(3z)\eta(6z)}+9\frac{[\eta(3z)\eta(6z)]^{3}}{\eta(z)\eta(2z)},
\end{align}so we have a weight-4 modular form\begin{align}&[Z_{6,3}(z)]^{2}X_{6,3}(z)\sqrt{\smash[b]{1+4X_{6,3}(z)}}\sqrt{\smash[b]{1+16X_{6,3}(z)}}\notag\\={}&
\left[\frac{\eta (2 z) \eta (6 z)}{\eta (z) \eta (3 z)}\right]^3 \left\{[\eta (z) \eta (2 z)]^4+9[ \eta (3 z)\eta (6 z)]^4\right\},\label{eq:sigma_wt4}
\end{align}as given in the integrands of \eqref{eq:KKKK_Hankel_eta} and \eqref{eq:IKKKK_Eichler}.

In addition to a routine analytic continuation, we need to check two more things for the extension of our modular parametrization to $ x\in[2,4]$.

First, we  show that the modular function $ X_{6,3}(z)=\left[\frac{2 \eta (2z ) \eta (6 z )}{\eta (z ) \eta (3z)}\right]^{6}$ is real-valued along the geodesic segment    $ z=\frac12+\frac{i}{2\sqrt{3}}e^{i\varphi},\varphi\in[0,\pi/3]$. From an analytic continuation of the last line in \eqref{eq:X63_symm}, it is clear  that $ X_{6,3}\big( \frac12+\frac{i}{2\sqrt{3}}e^{i\varphi} \big)= X_{6,3}\big( \frac12+\frac{i}{2\sqrt{3}}e^{-i\varphi} \big)$. By modular invariance with respect to $z\mapsto z-1 $, we see that the same expression is also equal to $ X_{6,3}\big( -\frac12+\frac{i}{2\sqrt{3}}e^{-i\varphi} \big)=\overline{X_{6,3}\big( \frac12+\frac{i}{2\sqrt{3}}e^{i\varphi} \big)}$, its own complex conjugate.

Then, by modifying our arguments in the second half of Lemma \ref{lm:X63}, we can check that  $ X_{6,3}:\big\{\frac12+\frac{i}{2\sqrt{3}}e^{i\varphi}\big|\varphi\in[0,\pi/3]\big\}\longrightarrow\big[-\frac14,-\frac1{16}\big]$
 is bijective.
 \end{proof}\begin{remark}In the proposition above, our modular parametrizations of the motivic integral $ \mathscr I(u)$ differ from the Bloch--Kerr--Vanhove approach \cite[][(2.3.44)]{BlochKerrVanhove2015}, but closely resemble certain Eichler integrals in our previous work \cite[][\S4]{EZF} that served as precursors to Epstein zeta functions. In fact, the only methodological innovation here is that we are now working with  Eichler integrals on   $ \varGamma_0(6)_{+3}$, rather than on the simpler  Hecke congruence group   $ \varGamma_0(4)$, as  in \cite[][\S4]{EZF}. We refer our readers to \cite[][\S2]{AGF_PartII} for more arithmetic applications of inhomogeneous Picard--Fuchs equations.   \eor\end{remark}

\subsection{Special values of Eichler integrals\label{subsec:BM_det_5Bessel}}
If we want to compute the integral $\IKM(1,4;2k+1)= \int_0^{\infty}I_0(t)[K_0(t)]^4t^{2k+1}\D t$ for  $ k\in\{1,2\}$,  we may apply the differential identity in \eqref{eq:Bessel_diff_trick} to the Eichler integral representation in  \eqref{eq:IKKKK_Eichler}, at $ z=\frac12+\frac{i\sqrt{5}}{2\sqrt{3}}$. As we have closed-form evaluations of $ X_{6,3}(z)$, $Z_{6,3}(z) $ and their derivatives at this specific CM point    in Table~\ref{tab:spec_X63_Z63}, the remaining challenge resides in the computation of the Eichler integral\begin{align}
\mathscr E(z):={}&12\pi ^{3}i\int^{i\infty}_z\left[\frac{\eta (2 z') \eta (6 z')}{\eta (z') \eta (3 z')}\right]^3 \left\{[\eta (z') \eta (2 z')]^4+9[ \eta (3 z')\eta (6 z')]^4\right\}(z-z')^2\D z'+\frac{7\zeta(3)}{8}\notag\\={}&\frac{1}{Z_{6,3}(z)}\int_0^\infty I_{0}\big(8\sqrt{\smash[b]{-X_{6,3}}(z)}t\big)[K_0(t)]^4t\D t,
\intertext{along with its derivatives}\mathscr E'(z):={}&24\pi ^{3}i\int^{i\infty}_z\left[\frac{\eta (2 z') \eta (6 z')}{\eta (z') \eta (3 z')}\right]^3 \left\{[\eta (z') \eta (2 z')]^4+9[ \eta (3 z')\eta (6 z')]^4\right\}(z-z')\D z',\label{eq:Eichler'_defn}\\\mathscr E''(z):={}&24\pi ^{3}i\int^{i\infty}_z\left[\frac{\eta (2 z') \eta (6 z')}{\eta (z') \eta (3 z')}\right]^3 \left\{[\eta (z') \eta (2 z')]^4+9[ \eta (3 z')\eta (6 z')]^4\right\}\D z',\label{eq:Eichler''_defn}\end{align}at   $ z=\frac12+\frac{i\sqrt{5}}{2\sqrt{3}}$.
Meanwhile, special values of higher-order derivatives, such as \begin{align}
\mathscr E'''\left( \frac{1}{2}+\frac{i\sqrt{5}}{2\sqrt{3}} \right)={}&27 i \sqrt{5} \pi  c^2,&\mathscr E''''\left( \frac{1}{2}+\frac{i\sqrt{5}}{2\sqrt{3}} \right)={}&-108 \sqrt{3} \pi  c^2 (3 c+1),
\end{align}\big[with $c=\frac{1}{240\pi^2}\Gamma \left(\frac{1}{15}\right) \Gamma \left(\frac{2}{15}\right) \Gamma \left(\frac{4}{15}\right) \Gamma \left(\frac{8}{15}\right)$\big] are readily computable from  the expression [see \eqref{eq:sigma_wt4} and \eqref{eq:Eichler''_defn}]\begin{align} \mathscr E'''(z)=-24\pi ^{3} i[Z_{6,3}(z)]^{2}X_{6,3}(z)\sqrt{\smash[b]{1+4X_{6,3}(z)}}\sqrt{\smash[b]{1+16X_{6,3}(z)}},\end{align} and entries in  Table~\ref{tab:spec_X63_Z63}.

\begin{lemma}[Special values of $ \mathscr E(z)$ and $ \mathscr E'(z)$]\label{lm:Eichler_deriv}We have the following identities:\begin{align}
\mathscr E\left( \frac{1}{2}+\frac{i\sqrt{5}}{2\sqrt{3}} \right)={}&\frac{\pi^{3}}{8\sqrt{15}},\label{eq:Eichler_spec}\\
\mathscr E'\left( \frac{1}{2}+\frac{i\sqrt{5}}{2\sqrt{3}} \right)={}&\frac{\pi ^3}{20 i}-\frac{3 \pi\IKM(0,3;1) }{2 \sqrt{5} i}.\label{eq:Eichler'_spec}\end{align}\end{lemma}\begin{proof}The evaluation in \eqref{eq:Eichler_spec} comes from Theorem \ref{thm:Bologna} and the special value for $Z_{6,3}\bigg(\frac12+\frac{i\sqrt{5}}{2\sqrt{3}}\bigg)$ in Table~\ref{tab:spec_X63_Z63}.

 Before computing  $ \mathscr E'(z)$ at $ z=\frac12+\frac{i\sqrt{5}}{2\sqrt{3}}$, we need to consider \begin{align}
\left.\frac{\D}{\D x}\right|_{x=1}\int_0^\infty  I_0(xt)[K_0(t)]^4t\D t=\int_0^\infty  I_1(t)[K_0(t)]^4t^2\D t.
\end{align}Integrating by parts, we obtain\begin{align}\int_0^\infty  I_1(t)[K_0(t)]^4t^2\D t=-2\int_{0}^\infty I_0(t)[K_0(t)]^4t\D t+4\int_0^\infty  I_0(t)K_{1}(t)[K_0(t)]^3t^2\D t.\end{align}Using the Wro\'nskian relation $I_0(t)K_1(t)+I_1(t)K_0(t)=1/t $, we get\begin{align}
\int_0^\infty  I_1(t)[K_0(t)]^4t^2\D t={}&\frac{4}{5}\int_0^\infty  [K_0(t)]^3t\D t-\frac{2}{5}\int_{0}^\infty I_0(t)[K_0(t)]^4t\D t\notag\\={}&\frac{2}{5}[2\IKM(0,3;1)-\IKM(1,4;1)].\label{eq:I1I0}
\end{align}

At the point $ z=\frac12+\frac{i\sqrt{5}}{2\sqrt{3}}$ where $ X_{6,3}(z)=-\frac{1}{64}$, we differentiate both sides of \begin{align}&
\int_0^\infty I_{0}\big(8\sqrt{\smash[b]{-X_{6,3}}(z)}t\big)[K_0(t)]^4t\D t=Z_{6,3}(z)\mathscr E(z)\label{eq:Eichler_repn_X63Z63}
\end{align}in $ z$, to deduce, respectively,\begin{align}
&-32X'_{6,3}(z)\int_0^\infty  I_1(t)[K_0(t)]^4t^2\D t\notag\\={}&3\sqrt{15}ic\int_0^\infty  I_1(t)[K_0(t)]^4t^2\D t=\frac{6\sqrt{15}ic}{5}[2\IKM(0,3;1)-\IKM(1,4;1)]
\end{align}and\begin{align}&Z_{6,3}'(z)\mathscr E(z)+Z_{6,3}(z)\mathscr E'(z)=-\frac{2\sqrt{3} \pi ^2 ic (3 c-1)}{5}  +\frac{8\sqrt{3}c}{\pi}\mathscr E'\left( \frac{1}{2}+\frac{i\sqrt{5}}{2\sqrt{3}} \right),
\end{align}where $ c=\sqrt{5}C=\frac{1}{240\pi^2}\Gamma \left(\frac{1}{15}\right) \Gamma \left(\frac{2}{15}\right) \Gamma \left(\frac{4}{15}\right) \Gamma \left(\frac{8}{15}\right)=\sqrt{5}\IKM(1,4;1)/\pi^2$ is the ``rescaled Bologna constant'' introduced in  Table~\ref{tab:spec_X63_Z63}. Comparing the last two displayed equations, we arrive at the  value of $ \mathscr E'\left( \frac{1}{2}+\frac{i\sqrt{5}}{2\sqrt{3}} \right)$ given in \eqref{eq:Eichler'_spec}.
\end{proof}\begin{lemma}[A special value of $ \mathscr E''(z)$]\label{lm:Eichler''}We have the following identity:\begin{align}
240\int^{i\infty}_{\frac{1}{2}+\frac{i\sqrt{5}}{2\sqrt{3}}}\left[\frac{\eta (2 z) \eta (6 z)}{\eta (z) \eta (3 z)}\right]^3 \left\{[\eta (z) \eta (2 z)]^4+9[ \eta (3 z)\eta (6 z)]^4\right\}(2z-1)\D z=1,\label{eq:id240}
\end{align} which entails\begin{align}\begin{split}
\mathscr E''\left( \frac{1}{2}+\frac{i\sqrt{5}}{2\sqrt{3}} \right)={}&\frac{3\sqrt{3}\pi}{5}\IKM(0,3;1).\label{eq:Eichler''_spec}
\end{split}\end{align}\end{lemma}\begin{proof}Upon comparison between \eqref{eq:xz_Jac} and \eqref{eq:sigma_wt4}, we see that \begin{align}
\left[\frac{\eta (2 z) \eta (6 z)}{\eta (z) \eta (3 z)}\right]^3 \left\{[\eta (z) \eta (2 z)]^4+9[ \eta (3 z)\eta (6 z)]^4\right\}=\frac{Z_{6,3}(z)}{2\pi i}\frac{\D X_{6,3}(z)}{\D z}.
\end{align}Integrating  \eqref{eq:JJJJ_Hankel_eta}, namely\begin{align}
\int_0^\infty J_{0}\big(8\sqrt{\smash[b]{-X_{6,3}}(z)}t\big)[J_0(t)]^4t\D t=\frac{3(2z-1)}{4\pi i}Z_{6,3}(z)
\end{align}over the differential $ \D X_{6,3}(z)$, we identify the left-hand side of \eqref{eq:id240} with \begin{align}
5\int_0^\infty J_1(t)[J_0(t)]^4\D t=-\int_0^\infty \frac{\D }{\D t}[J_0(t)]^5\D t=1.
\end{align}

Meanwhile, the integral representations in  \eqref{eq:Eichler'_defn} and \eqref{eq:Eichler''_defn} tell us that the left-hand side of \eqref{eq:id240} is also equal to \begin{align}
\frac{20i}{\pi^{3}}\left[ \mathscr E'\left( \frac{1}{2}+\frac{i\sqrt{5}}{2\sqrt{3}} \right)-\frac{i\sqrt{5}}{2\sqrt{3}} \mathscr E''\left( \frac{1}{2}+\frac{i\sqrt{5}}{2\sqrt{3}} \right)\right].
\end{align} This verifies \eqref{eq:Eichler''_spec}.\end{proof}\begin{theorem}[$ \IKM(1,4;3) $  and $ \IKM(1,4;5)$  via  $  \mathscr E(z)$, $  \mathscr E'(z)$ and $ \mathscr E''(z)$]\label{prop:IKM143_145}We have\begin{align}
\IKM(1,4;3)={}&\pi^{2}\left( \frac{2}{15} \right)^2\left( 13C-\frac{1}{10C} \right),&\IKM(1,4;5)={}&\pi^{2}\left( \frac{4}{15} \right)^3\left( 43C-\frac{19}{40C} \right),
\end{align}where  $ C=\frac{1}{240 \sqrt{5}\pi^{2}}\Gamma \left(\frac{1}{15}\right) \Gamma \left(\frac{2}{15}\right) \Gamma \left(\frac{4}{15}\right) \Gamma \left(\frac{8}{15}\right)$ is the ``Bologna constant''. \end{theorem}\begin{proof}As we twice differentiate \eqref{eq:Eichler_repn_X63Z63} with respect to $z$, and set $ X_{6,3}(z)=-\frac{1}{64}$ afterwards, we obtain a formula\begin{align}{}&
-32(64X'^{2}+X'')\int_{0}^\infty I_1(t)[K_0(t)]^4t^2\D t+1024X'^{2}\IKM(1,4;3)\notag\\={}&Z''\mathscr E+2Z'\mathscr E'+Z\mathscr E'',
\end{align} where the subscripts for $X_{6,3}$ and $ Z_{6,3}$ are dropped, and the argument  $ z=\frac12+\frac{i\sqrt{5}}{2\sqrt{3}}$  is suppressed throughout, to save space. Substituting known results from   Table~\ref{tab:spec_X63_Z63} and Lemmata \ref{lm:Eichler_deriv}--\ref{lm:Eichler''}, we may transcribe the last equality into \begin{align}
135c^{2}\IKM(1,4;3)=\frac{6 \sqrt{5} \pi ^2c(26 c^2-1)}{25} ,
\end{align}which confirms the evaluation for $ \IKM(1,4;3)$.

Taking fourth-order derivatives on  \eqref{eq:Eichler_repn_X63Z63}, we arrive at {\allowdisplaybreaks\begin{align}&-32\left[ X''''+64 (3 X''^2+24576 X'^4+4 X''' X'+768 X'^2 X'') \right]\int_{0}^\infty I_1(t)[K_0(t)]^4t^2\D t\notag\\{}&+1024(3 X''^2+24576 X'^4+4 X''' X'+768 X'^2 X'')\IKM(1,4;3)\notag\\{}&-65536X'^2 (3 X''+128 X'^2)\int_{0}^\infty I_1(t)[K_0(t)]^4t^4\D t+1048576X'^{4}\IKM(1,4;5)\notag\\
={}&Z''''\mathscr E+4Z'''\mathscr E'+6Z''\mathscr E''+4Z'\mathscr E'''+Z\mathscr E'''',\label{eq:4th_deriv}
\end{align}}where\begin{align}
\int_{0}^\infty I_1(t)[K_0(t)]^4t^4\D t=\frac{4}{5}[\IKM(0,3;3)-\IKM(1,4;3)]
\end{align} can be derived in a similar vein as  \eqref{eq:I1I0}, and  the relation $ \IKM(0,3;3)=2[2\IKM(0,3;1)-1]/3$ has been proved in \cite[][\S3.2]{Bailey2008}.
Now that the left-hand side of \eqref{eq:4th_deriv} equals \begin{align}
&-\frac{648c (78 c^3-36 c^2+18 c-1)}{5} \times \frac{2}{5}\left[2\IKM(0,3;1)-\frac{\pi^{2}c}{\sqrt{5}}\right]\notag\\{}&-2916 c^2 (3 c+1) (5 c-1)\IKM(1,4;3)\notag\\{}&-14580 (c-1) c^3\times\frac{4}{15}[4\IKM(0,3;1)-3\IKM(1,4;3)-2]+18225 c^4\IKM(1,4;5)
\end{align} and its right-hand side amounts to\begin{align}&\frac{216 \pi ^2 c (1330 c^4-684 c^3+124 c^2+12 c-3)}{25 \sqrt{5}}+7776 (c-1) c^3
\notag\\{}&-\frac{2592c (228 c^3-186 c^2+18 c-1)}{25}  \IKM(0,3;1),
\end{align}we can simplify the relation above into\begin{align}
&-729 c^{2} [4 (11 c^2+6 c-1) \IKM(1,4;3)-25 c^2 \IKM(1,4;5)]\notag\\={}&\frac{216 \pi ^2  c(862 c^4-468 c^3+16 c^2+18 c-3)}{25 \sqrt{5}}.
\end{align}This confirms the evaluation for $ \IKM(1,4;5)$. (Furthermore, based on the recursion for the  rescaled moments $   \IKM(1,4;2n+1)/\pi^{2},n\in\mathbb Z_{\geq0}$ \cite[][(11)]{Bailey2008}, one can show that all of them are rational combinations of $C$ and $1/C$.) \end{proof}
\begin{remark}We have recently found \cite[][\S2]{Zhou2017PlanarWalks} that the closed-form evaluation of  $ \IKM(1,4;3)$ can also be deduced from a result of Borwein--Straub--Vignat \cite[][Theorem 4.17]{BSV2016}, using Wick rotations.  \eor\end{remark}

\begin{remark}It is also possible to use factorizations of Wro\'nskians to compute the determinant of the matrix in \eqref{eq:BM_m2}, without evaluating the four individual Bessel moments. Such an algebraic approach is described in our recent manuscript \cite[][\S2]{Zhou2017BMdet}.\eor\end{remark}

\section{Feynman diagrams with 6  Bessel factors\label{sec:6Bessel}}
\subsection{Modular parametrization for certain Hankel transforms}Instead of working directly on the modularity of Feynman integrals with 6 Bessel factors, we will first analyze a small building block with 4 Bessel factors. The latter problem can be solved using the classical elliptic integrals  \cite[cf.][\S13.46, (9)]{Watson1944Bessel}, whose modular parametrization will be our major concern.  \begin{lemma}[Some Wick rotations]\label{lm:6BesselWick}\begin{enumerate}[leftmargin=*,  label=\emph{(\alph*)},ref=(\alph*),
widest=a, align=left]\item The following identities hold:
\begin{align}
\int_0^\infty [I_0(t)]^{2}[K_0(t)]^4t\D t={}&\frac{\pi^{4}}{30}\int_0^\infty [J_0(x)]^6x\D x,\label{eq:IIKKKK_JJJJJJ}\\\int_0^\infty I_0(t)[K_0(t)]^5t\D t={}&-\frac{\pi^5}{12}\int_0^\infty [J_0(x)]^5Y_{0}(x)x\D x.\label{eq:IKKKKK_JJJJJY}
\end{align}\item For $ x\in[0,1)$, we have \begin{align}
\int_0^\infty I_0(xt)I_0(t)[K_0(t)]^2t\D t=\frac{\pi^{2}}{6}\int_0^\infty J_0(xt)[J_{0}(t)]^{3}t\D t.\label{eq:IIKK_JJJJ}
\end{align}\item For $ x\in[0,3)$, we have\begin{align}
\int_0^\infty I_0(xt)[K_0(t)]^3t\D t={}&-\frac{\pi^{3}}{8}\int_0^\infty J_0(xt)Y_0(t)\{3[J_0(t)]^2-[Y_0(t)]^2\}t\D t,\label{eq:IKKK_JY_etc}\\3\int_0^\infty K_0(xt)I_{0}(t)[K_0(t)]^2t\D t={}&-\frac{\pi^{3}}{8}\int_0^\infty J_0(xt)Y_0(t)\{3[J_0(t)]^2+[Y_0(t)]^2\}t\D t\notag\\{}&-\frac{\pi^{3}}{4}\int_0^\infty Y_0(xt)[J_0(t)]^{3}t\D t.\label{eq:KIKK_JY_etc}
\end{align}\end{enumerate}\end{lemma}\begin{proof}\begin{enumerate}[leftmargin=*,  label=(\alph*),ref=(\alph*),
widest=a, align=left]\item As in the proof of Theorem \ref{thm:Bologna}, we compute\begin{align}
\left(\frac{2}{\pi}\right)^4\int_0^\infty [I_0(t)]^{2}[K_0(t)]^4t\D t={}& -\int_0^{i\infty} [J_0(z)]^{2}[H_0^{(1)}(z)]^4z\D z\notag\\={}&-\R\int_0^{\infty} [J_0(x)]^{2}[H_0^{(1)}(x)]^4x\D x\notag\\={}&-\int_0^{\infty} J^{2}(J^4-6 J^2 Y^2+Y^4)x\D x,
\end{align}where $ J=J_0(x),Y=Y_0(x)$ in the last step. Applying Lemma \ref{lm:BHJ} to \begin{align}&-J^2 (J^4-6 J^2 Y^2+Y^4)\notag\\{}&+\frac{J}{10} [(J+i Y)^5-(-J+i Y)^5]+
\frac{2J^3}{3}  [(J+i Y)^3-(-J+i Y)^3]=\frac{8 J^6}{15},
\end{align}we arrive at  \eqref{eq:IIKKKK_JJJJJJ}.

The proof of \eqref{eq:IKKKKK_JJJJJY} is essentially similar. \item By Jordan's lemma, we can justify the following Wick rotation for $x\in[0,1) $: \begin{align}
\left(\frac{2}{\pi}\right)^2\int_0^\infty I_0(xt)I_0(t)[K_0(t)]^2t\D t={}&\int_0^{i\infty} J_0(xz)J_0(z)[H_0^{(1)}(z)]^2z\D z\notag\\={}&\R\int_0^{\infty} J_0(xt)J_0(t)[H_0^{(1)}(t)]^2t\D t\notag\\={}&\int_0^\infty J_0(xt)J(J^{2}-Y^2)t\D t,
\end{align}where $J=J_0(t),Y=Y_0(t)$ in the last expression. Meanwhile, by a variation on Lemma \ref{lm:BHJ}, we have\begin{align}
\int_0^\infty J_0(xt)\frac{(J+iY)^3-(-J+iY)^3}{2}t\D t=\int_0^\infty J_0(xt)J(J^{2}-3Y^{2})t\D t=0,\label{eq:JJ_3YY}
\end{align}so the claimed identity is proved.\item To show \eqref{eq:IKKK_JY_etc}, simply take a Wick rotation:\begin{align}
\left(\frac{2}{\pi}\right)^3\int_0^\infty I_0(xt)[K_0(t)]^3t\D t={}&-\I\int_0^{i\infty}J_0(xz)[H_0^{(1)}(z)]^3z\D z\notag\\={}&-\I\int_0^{\infty}J_0(xt)[H_0^{(1)}(t)]^3t\D t\notag\\={}&-\int_0^\infty J_0(xt)Y(3J^2-Y^2)t\D t,
\end{align}where we use the abbreviation $J=J_0(t),Y=Y_0(t)$ as before.

For  \eqref{eq:KIKK_JY_etc}, Wick rotation alone brings us\begin{align}&
\left(\frac{2}{\pi}\right)^3\int_0^\infty K_0(xt)I_{0}(t)[K_0(t)]^2t\D t\notag\\={}&-2\int_0^\infty J_0(xt)J^{2}Yt\D t-\int_0^\infty Y_0(xt)J(J^{2}-Y^{2})t\D t.\label{eq:KIKK_Wick_prep}
\end{align}In the meantime, we extend the technique in Lemma \ref{lm:BHJ} to\begin{align}
\int_{0}^{\infty}\frac{[J_{0}(xt)+iY_0(xt)](J+iY)^3-[-J_{0}(xt)+iY_0(xt)](-J+iY)^3}{2i}t\D t=0,
\end{align}which implies \begin{align}
\int_0^\infty J_0(xt)Y(3J^2-Y^2)t\D t+\int_0^\infty Y_0(xt)J(J^2-3Y^2)t\D t=0.
\end{align} The equation above allows us to eliminate the term $ \int_0^\infty Y_0(xt)JY^2t\D t$ from \eqref{eq:KIKK_Wick_prep} and arrive at the right-hand side of   \eqref{eq:KIKK_JY_etc}.\qedhere\end{enumerate}\end{proof}Let $ h(x)=\int_0^\infty J_0(xt)I_0(t)[K_0(t)]^2t\D t$ be the Hankel transform of the function $ I_0(t)[K_0(t)]^2$, and $ \widetilde h(x)=\int_0^\infty J_0(xt)[J_0(t)]^3t\D t$ be a ``random walk integral'' ($ \widetilde h(x)=p_3(x)/x$, where $p_3(x)$ is the radial probability density of  the distance travelled by a random walker in the plane, taking three consecutive steps of unit lengths). According to  the Parseval--Plancherel theorem for Hankel transforms  \cite[cf.][(16)]{Bailey2008}, we have\begin{align}
\int_0^\infty [I_0(t)]^{2}[K_0(t)]^4t\D t=\int_{0}^\infty[h(x)]^2x\D x,\quad \int_0^\infty[J_{0}(t)]^6t\D t=\int_{0}^\infty[\widetilde h(x)]^2x\D x.\label{eq:IIKKKK_PPT}
\end{align}In order to  recast the left-hand sides of the equations above into  Eichler integrals, we need to represent the Hankel transforms  $h(x)$ and $\widetilde h(x) $ as modular forms.\begin{proposition}[Modular parametrizations of two Hankel transforms]\label{prop:Hankel_mod}\begin{enumerate}[leftmargin=*,  label=\emph{(\alph*)},ref=(\alph*),
widest=a, align=left]\item For $x>0$, we have a hypergeometric evaluation \begin{align}
\int_0^\infty J_0(xt)I_0(t)[K_0(t)]^2t\D t=\frac{\pi}{\sqrt{3} } \frac{1}{3+x^2} \, _2F_1\left(\left.\begin{array}{c}
\frac{1}{3},\frac{2}{3} \\1 \\
\end{array}\right|\frac{x^4 (9+x^2)}{(3+x^2)^3}\right),\label{eq:JIKK_2F1}
\end{align}which can be parametrized as \begin{align}
\int_0^\infty J_0\left(i\left[\frac{\theta\left(1-\frac{1}{3w}\right)}{\theta\left(3-\frac{1}{w}\right)}\right]^{2}t\right)I_0(t)[K_0(t)]^2t\D t=\frac{\pi}{3\sqrt{3}}\frac{\eta (3 w)[ \eta (2 w)]^6}{[\eta (w)]^3 [\eta (6 w)]^2},\label{eq:IKK_Hankel_mod}
\end{align}where $\theta(z):=\sum_{n\in\mathbb Z}e^{\pi i n^2z} $ is one of  Jacobi's elliptic theta functions (``Thetanullwerte''), and $ w=-\frac{1}{2}+iy$ for $ y>0$.\item For $ x\in(0,1)$, the function $p_3(x)/x= \int_0^\infty J_0(xt)[J_0(t)]^3t\D t$ admits a modular parametrization\begin{align}
\int_0^\infty J_0\left(\left[\frac{\theta\left(1-\frac{1}{3w}\right)}{\theta\left(3-\frac{1}{w}\right)}\right]^{2}t\right)[J_0(t)]^3t\D t={}&\frac{2}{\sqrt{3}\pi}\frac{\eta (3 w)[ \eta (2 w)]^6}{[\eta (w)]^3 [\eta (6 w)]^2},\label{eq:p3_small}
\intertext{where $w/i>0$; for $ x\in(1,3)$, the function $ p_3(x)/x$ can be parametrized as}
\int_0^\infty J_0\left(\left[\frac{\theta\left(1-\frac{1}{3w}\right)}{\theta\left(3-\frac{1}{w}\right)}\right]^{2}t\right)[J_0(t)]^3t\D t={}&\frac{2(1-3w)}{\sqrt{3}\pi}\frac{\eta (3 w)[ \eta (2 w)]^6}{[\eta (w)]^3 [\eta (6 w)]^2},\label{eq:p3_medium}
\end{align}where $ w=(1+e^{i\varphi})/6,\varphi\in(0,\pi)$; for $x>3$, we have $ p_3(x)/x=0$.\end{enumerate}\end{proposition}\begin{proof}\begin{enumerate}[leftmargin=*,  label=(\alph*),ref=(\alph*),
widest=a, align=left]\item For sufficiently small $x$, we have \begin{align}
\int_0^\infty J_0(xt)I_0(t)[K_0(t)]^2t\D t=\frac{\pi}{\sqrt{3} } \frac{1}{3-x^2} \, _2F_1\left(\left.\begin{array}{c}
\frac{1}{3},\frac{2}{3} \\1 \\
\end{array}\right|-\frac{x^2 (9+x^2)^2}{(3-x^2)^3}\right),\label{eq:JIKK_prep}
\end{align}by the Wick rotation in \eqref{eq:IIKK_JJJJ} and an  analytic continuation of the hypergeometric representation for $ \int_0^\infty J_0(xt)[J_0(t)]^3t\D t$ \cite[][(3.4)]{BSWZ2012}. Setting $p=-\frac{2 x^2}{x^2+3}$ in the following hypergeometric identity \cite[][Chap.\ 33, Theorem 6.1]{RN5}:
\begin{align}
{_2F_1}\left(\left.\begin{array}{c}\frac{1}{3},\frac{2}{3} \\1 \\
\end{array}\right|\frac{p (3+p)^2}{2 (1+p)^3}\right)
=(1+p){_
2F_1}\left(\left.\begin{array}{c}
\frac{1}{3},\frac{2}{3} \\1 \\
\end{array}\right|\frac{p^{2} (3+p)}{4}\right),\label{eq:doubling_Hecke3}
\end{align}we  recast \eqref{eq:JIKK_prep} into \eqref{eq:JIKK_2F1}. The validity of  \eqref{eq:JIKK_2F1} extends to all $x>0$, by analytic continuation.

With a substitution $ x=i\big[\theta\big(-\frac{2}{3z}-1\big)\big/\theta\big(-\frac{2}{z}-3\big)\big]^{2}$, one can verify \begin{align}
\frac{x^4 (9+x^2)}{(3+x^2)^3}=\left\{ 1+\frac{1}{27}\left[ \frac{\eta(z)}{\eta(3z)} \right] ^{12}\right\}^{-1}
\end{align}by showing that the ratio between the two sides defines a bounded function on  $ X_{0}(3)=\varGamma_0(3)\backslash(\mathfrak H\cup\mathbb Q\cup\{i\infty\})$, and that the leading order $ q$-expansions of both sides agree. One can also check that the geodesic $ z=(5+e^{i\varphi})/12,\varphi\in(0,\pi)$ is mapped bijectively to $x\in(0,\infty)$, using a method similar to what was employed in the proof of Proposition~\ref{prop:Eichler_Iu}.

Meanwhile,  with the aforementioned relation between  $x\in(0,\infty)$ and $z=(5+e^{i\varphi})/12$ for $ \varphi\in(0,\pi)$, we  paraphrase an identity \cite[][Chap.~33, Corollary 3.4]{RN5} from Ramanujan's notebook as follows:\begin{align}
\frac{\sqrt{3} \sqrt[3]{\vphantom{1}x} \sqrt[4]{1+x^2} \sqrt[12]{9+x^2}}{3+x^2} \, _2F_1\left(\left.\begin{array}{c}
\frac{1}{3},\frac{2}{3} \\1 \\
\end{array}\right|\frac{x^4 (9+x^2)}{(3+x^2)^3}\right)=\left[\eta\left( \frac{2 z-1}{3 z-1} \right)\right]^{2}.
\end{align}  Multiplying both sides with  \begin{align} \frac{\sqrt3 \sqrt[12]{1+x^2}}{ \sqrt[3]{\vphantom{1}x}\sqrt[12]{9+x^2}}=\frac{\eta(z)}{\eta(3z)}=\frac{\eta\left( \frac{2 z-1}{3 z-1} \right)}{\eta\left( \frac{6 z-3}{3 z-1} \right)},\quad\text{where }\left(\begin{smallmatrix}2&-1\\3&-1\end{smallmatrix}\right)\in\varGamma_0(3),\end{align}we obtain\begin{align}
\frac{3\sqrt[3]{1+x^{2}}}{3+x^2}{_2F_1}\left(\left.\begin{array}{c}
\frac{1}{3},\frac{2}{3} \\1 \\
\end{array}\right|\frac{x^4 (9+x^2)}{(3+x^2)^3}\right)=\frac{\left[\eta\left( \frac{2 z-1}{3 z-1} \right)\right]^3}{\eta\left( \frac{6 z-3}{3 z-1} \right)}.
\end{align}  Furthermore, by a theta function identity \cite[][Chap.~18, (24.31)]{RN3} in Ramanujan's notebook, we have \begin{align}\sqrt[3]{1+
x^{2}}=\sqrt[3]{1-\left[\frac{\theta\left( -\frac{2}{3z}-1 \right)}{\theta\left( -\frac{2}{z}-3 \right)}\right]^{4}}=1-\frac{\theta\left( -\frac{2}{9z}-\frac{1}{3} \right)}{\theta\left( -\frac{2}{z}-3 \right)},\label{eq:theta_cubic_id}
\end{align}and the last expression can be reduced by an identity\begin{align}
1-\frac{\theta\left( \frac{2\tau}{3} -1\right)}{\theta(6\tau-9)}=2\frac{\eta(\tau)}{\eta(2\tau)}\left[ \frac{\eta(6\tau)}{\eta(3\tau)} \right]^{3},\quad \forall \tau\in\mathfrak H,\label{eq:theta_eta_x9}
\end{align}   also due to Ramanujan \cite[][Chap.~16, Entry 24(iii) and Chap.~20, Entry 1(ii)]{RN3}.

Finally, setting $ \tau=1-\frac{1}{3z}$ and $ \frac{2 z-1}{3 z-1}=1+2w\in i\mathbb R$ for  $ z=(5+e^{i\varphi})/12,\varphi\in(0,\pi)$, while simplifying eta functions with the modular transformation $ \eta(-1/\tau')=\sqrt{\tau'/i}\eta(\tau')$ where necessary, we arrive at the expression in \eqref{eq:IKK_Hankel_mod}.\item The modular parametrization in \eqref{eq:p3_small} follows directly from analytic continuation of  \eqref{eq:IKK_Hankel_mod} and the Wick rotation relation in \eqref{eq:IIKK_JJJJ}.

One notes that the smooth functions $ p_3(x),x\in(0,1)$ and $ p_3(x),x\in(1,3)$ satisfy the same ordinary differential equation of second order  \cite[][Theorem 2.4]{BSWZ2012}, so  $ p_3(x)/x,x\in(1,3)$ must be a linear combination of \begin{align}
\frac{\eta (3 w)[ \eta (2 w)]^6}{[\eta (w)]^3 [\eta (6 w)]^2}\quad\text{and}\quad \frac{w\eta (3 w)[ \eta (2 w)]^6}{[\eta (w)]^3 [\eta (6 w)]^2}
\end{align}for $x=\left[\theta\left(1-\frac{1}{3w}\right)\big/\theta\left(3-\frac{1}{w}\right)\right]^{2} $. Here, the linear combination must be proportional to $ (1-3w)$, so as to guarantee finiteness of $ p_3(x)/x$ in the $ x\to3-0^+$ regime. The precise prefactor can be determined by asymptotic analysis of $ p_3(x)/x$ and $q$-expansion of the modular form. This proves \eqref{eq:p3_medium}.

For $x>3$, one can prove  $ \int_0^\infty J_0(xt)[J_0(t)]^3t\D t=0$ by extracting the real part from the following Wick rotation:\begin{align}
\int_0^\infty H_{0}^{(1)}(xt)[J_0(t)]^3t\D t=\frac{2i}{\pi}\int_0^\infty [I_0(t)]^3K_{0}(xt)t\D t,\quad \forall x>3.
\end{align}Alternatively, one may invoke the probabilistic interpretation of  $p_3(x)=\int_0^\infty J_0(xt)[J_0(t)]^3xt\D t$ to conclude that $ p_3(x)/x=0$ for $x>3$.    \qedhere\end{enumerate}    \end{proof}\begin{remark}The modular parametrizations in the proposition above are foreshadowed by the following formula (see  \cite[][\S13.46, (9)]{Watson1944Bessel} and \cite[][(3)]{BNSW2011}) for $ x\in(0,1)\cup(1,3)$:\begin{align}
\int_0^\infty J_0(xt)[J_{0}(t)]^{3}t\D t=\frac{1}{\pi^{2}\sqrt{x}}\R {_
2F_1}\left(\left.\begin{array}{c}
\frac{1}{2},\frac{1}{2} \\1 \\
\end{array}\right|\frac{(3-x) (1+x)^3}{16 x}\right),
\end{align}
and the fact that \cite[][Chap.~33, Lemma 5.5 and Theorem 5.6]{RN5}\begin{align}
 {_
2F_1}\left(\left.\begin{array}{c}
\frac{1}{2},\frac{1}{2} \\1 \\
\end{array}\right|-\frac{(3+t^{2}) (1-t^{2})^3}{16 t^{2}}\right)=[\theta(3z)]^{2},\quad\text{for } t=\frac{\theta(z)}{\theta(3z)},z/i>0.
\end{align}  Formally, we may regard \eqref{eq:IKK_Hankel_mod} as  an analytic continuation of the identities above, along with a modular transformation corresponding to \eqref{eq:doubling_Hecke3}.\eor\end{remark}\begin{remark}It is also possible to parametrize the aforementioned Hankel transforms without using Jacobi's theta functions. For example, after comparing the Taylor expansion of $ p_3(x),0\leq x<1$ due to Borwein--Straub--Wan--Zudilin \cite[][(3.2)]{BNSW2011} to Zagier's Ap\'ery-like recurrence (Case \textbf{C}) \cite{Zagier2009Apery}, we arrive at \begin{align}
\int_0^\infty J_0\left(\frac{3[ \eta (w)]^2 [\eta (6 w)]^4}{[\eta (3 w)]^2[\eta (2 w)]^4 }t\right)[J_0(t)]^3t\D t=\frac{2}{\sqrt{3}\pi}\frac{\eta (3 w)[ \eta (2 w)]^6}{[\eta (w)]^3 [\eta (6 w)]^2}\tag{\ref{eq:p3_small}$'$},
\end{align}for $w/i>0$, which is an alternative formulation of \eqref{eq:p3_small}. For yet another approach to this modular parametrization, see Broadhurst's recent talks at Vienna (\cite[][\S1.2]{Broadhurst2017DESY} and  \cite[][\S1.2]{Broadhurst2017ESIa}), which refers to his earlier talk at Les Houches \cite[][\S2.5]{Broadhurst2014LesHouches}.    \eor\end{remark}

In addition to the usual Hankel transform $ \int_0^\infty J_0(xt)f(t)t\D t$ of a function $ f(t),t\in(0,\infty)$, we will also need the $Y$-transform $ \int_0^\infty Y_0(xt)f(t) t\D t$ and the $K$-transform $ \int_0^\infty K_0(xt)f(t)t\D t$ for certain Bessel moments.
\begin{proposition}[$Y$- and $K$-transforms]\label{prop:Y_K_transforms}\begin{enumerate}[leftmargin=*,  label=\emph{(\alph*)},ref=(\alph*),
widest=a, align=left]\item We have \begin{align}
&\int_0^\infty J_0\left(i\left[\frac{\theta\left(1-\frac{1}{3w}\right)}{\theta\left(3-\frac{1}{w}\right)}\right]^{2}t\right)[K_0(t)]^3t\D t-\frac{3\pi}{2}\int_0^\infty Y_0\left(i\left[\frac{\theta\left(1-\frac{1}{3w}\right)}{\theta\left(3-\frac{1}{w}\right)}\right]^{2}t\right)I_0(t)[K_0(t)]^2t\D t\notag\\={}&\frac{\pi^{2}(2w+1)}{2\sqrt{3}i}\frac{\eta (3 w)[ \eta (2 w)]^6}{[\eta (w)]^3 [\eta (6 w)]^2},\label{eq:JKKK_YIKK}
\end{align}where $ w=-\frac{1}{2}+iy$ for $ y>0$, and\begin{align}&
\int_0^\infty I_0\left(\left[\frac{\theta\left(1-\frac{1}{3w}\right)}{\theta\left(3-\frac{1}{w}\right)}\right]^{2}t\right)[K_0(t)]^3t\D t+3\int_0^\infty K_0\left(\left[\frac{\theta\left(1-\frac{1}{3w}\right)}{\theta\left(3-\frac{1}{w}\right)}\right]^{2}t\right)I_0(t)[K_0(t)]^2t\D t\notag\\={}&\frac{\pi^{2}w}{\sqrt{3}i}\frac{\eta (3 w)[ \eta (2 w)]^6}{[\eta (w)]^3 [\eta (6 w)]^2} \label{eq:IKKK_KIKK}
\end{align}   for $w/i>0$.\item We have \begin{align}&
3\int_0^\infty J_0\left(\left[\frac{\theta\left(1-\frac{1}{3w}\right)}{\theta\left(3-\frac{1}{w}\right)}\right]^{2}t\right)[J_0(t)]^2Y_0(t)t\D t+\int_0^\infty Y_0\left(\left[\frac{\theta\left(1-\frac{1}{3w}\right)}{\theta\left(3-\frac{1}{w}\right)}\right]^{2}t\right)[J_0(t)]^3t\D t\notag\\={}&-\frac{4w}{\sqrt{3}\pi i}\frac{\eta (3 w)[ \eta (2 w)]^6}{[\eta (w)]^3 [\eta (6 w)]^2}\label{eq:JJJY_YJJJ}
\end{align}for  $w/i>0$ and  $ w=(1+e^{i\varphi})/6,\varphi\in(0,\pi)$. \end{enumerate}\end{proposition}\begin{proof}\begin{enumerate}[leftmargin=*,  label=(\alph*),ref=(\alph*),
widest=a, align=left]\item We observe that the sequences $ c_{3,k}:=\int_0^\infty [K_0(t)]^3t^k\D t$ and $ s_{3,k}:=\int_0^\infty I_{0}(t)[K_0(t)]^2t^k\D t$ satisfy the same recursion \cite[][(8)]{Bailey2008}, namely,  $ (k+1)^4 c_{3,k}-2(5k^2+20k+21)c_{3,k+2}+9c_{3,k+4}=0$ and $ (k+1)^4 s_{3,k}-2(5k^2+20k+21)s_{3,k+2}+9s_{3,k+4}=0$ both hold for non-negative integers $k$. As a result, the function\begin{align}
\int_0^\infty J_0(\sqrt{u}t)I_0(t)[K_0(t)]^2t\D t=\frac{\pi}{\sqrt{3} } \frac{1}{3+u} \, _2F_1\left(\left.\begin{array}{c}
\frac{1}{3},\frac{2}{3} \\1 \\
\end{array}\right|\frac{u^{2}(9+u)}{(3+u)^3}\right)
\end{align} is annihilated by the differential operator\begin{align}
\widehat B_3:=u (u+1) ( u+9)\frac{\D^2}{\D u^2}+(3 u^2 + 20 u + 9)\frac{\D}{\D u}+(u+3),\label{eq:defn_B3}
\end{align}and we have an inhomogeneous differential equation\begin{align}
\widehat B_3\left\{ \int_0^\infty J_0(\sqrt{u}t)[K_0(t)]^3t\D t \right\}=\frac{3}{2}.
\end{align}Meanwhile, differentiating under the integral sign and integrating by parts \cite[cf.][\S9]{Vanhove2014Survey}, we can verify that \begin{align}
\widehat B_3\left\{ \int_0^\infty Y_0(\sqrt{u}t)\pi I_{0}(t)[K_0(t)]^2t\D t \right\}=1.
\end{align}

In view of the analysis above, the left-hand side of \eqref{eq:JKKK_YIKK} must be equal to\begin{align}
\frac{\eta (3 w)[ \eta (2 w)]^6}{[\eta (w)]^3 [\eta (6 w)]^2}[ k_0 +k_1(2w+1)]
\end{align}where $k_0$ and $k_1$ are constants. Since  $ Y_0(xt)=\frac{2}{\pi}\log(xt)+O(1)$ as $ x\to0^+$, and $ \int_0^\infty I_0(t)[K_0(t)]^2t\D t=\frac{\pi}{3\sqrt{3}}$ \cite[][(23)]{Bailey2008}, we can determine $ k_1=\frac{\pi^{2}}{2\sqrt{3}i}$ immediately. Superimposing with \eqref{eq:IKK_Hankel_mod}, we obtain\begin{align}&
\int_0^\infty J_0\left(i\left[\frac{\theta\left(1-\frac{1}{3w}\right)}{\theta\left(3-\frac{1}{w}\right)}\right]^{2}t\right)[K_0(t)]^3t\D t-\frac{3\pi}{2i}\int_0^\infty H_0^{(1)}\left(i\left[\frac{\theta\left(1-\frac{1}{3w}\right)}{\theta\left(3-\frac{1}{w}\right)}\right]^{2}t\right)I_0(t)[K_0(t)]^2t\D t\notag\\={}&\frac{\eta (3 w)[ \eta (2 w)]^6}{[\eta (w)]^3 [\eta (6 w)]^2} \left(k_0 +\frac{\pi^{2}w}{\sqrt{3}i}\right),
\end{align}which analytically continues to\begin{align}&
\int_0^\infty I_0\left(\left[\frac{\theta\left(1-\frac{1}{3w}\right)}{\theta\left(3-\frac{1}{w}\right)}\right]^{2}t\right)[K_0(t)]^3t\D t+3\int_0^\infty K_0\left(\left[\frac{\theta\left(1-\frac{1}{3w}\right)}{\theta\left(3-\frac{1}{w}\right)}\right]^{2}t\right)I_0(t)[K_0(t)]^2t\D t\notag\\={}&\frac{\eta (3 w)[ \eta (2 w)]^6}{[\eta (w)]^3 [\eta (6 w)]^2} \left(k_0 +\frac{\pi^{2}w}{\sqrt{3}i}\right)
\end{align}for $ w/i>0$. Taking the $ w\to i0^+$ limit, and recalling the evaluation  $ \int_0^\infty I_0(t)[K_0(t)]^3t\D t=\pi^2/16$ from \cite[][(54)]{Bailey2008}, we find $k_0=0$.

Thus far, we have confirmed both  \eqref{eq:JKKK_YIKK}  and \eqref{eq:IKKK_KIKK}. \item We note that the expression $ \int_0^\infty I_0(xt)[K_0(t)]^3t\D t+3\int_0^\infty K_0(xt)I_0(t)[K_0(t)]^2t\D t$ is continuous with respect to $x\in(0,3)$, and the right-hand side of  \eqref{eq:IKKK_KIKK} is  smooth in a neighborhood of $ i0^+$. Therefore, the validity of  \eqref{eq:IKKK_KIKK} extends to the geodesic   $ w=(1+e^{i\varphi})/6,\varphi\in(0,\pi)$, by analytic continuation.

  Adding up \eqref{eq:IKKK_JY_etc} and \eqref{eq:KIKK_JY_etc}, we  derive \eqref{eq:JJJY_YJJJ} from  \eqref{eq:IKKK_KIKK}. \qedhere\end{enumerate}\end{proof}
\subsection{Eichler integrals via Hankel fusions}We can now use the modular parametrizations in Proposition \ref{prop:Hankel_mod}  to fuse  Hankel transforms into  Feynman integrals involving 6 Bessel factors, as planned in \eqref{eq:IIKKKK_PPT}.\begin{proposition}[Eichler formulation of $ \IKM(2,4;1)$]\label{prop:IKM241}We have \begin{align}
\int_0^\infty [I_0(t)]^{2}[K_0(t)]^4t\D t=\frac{\pi^{3}i}{3}\int_{-\frac{1}{2}}^{-\frac{1}{2}+i\infty}[\eta(w)\eta(2w)\eta(3w)\eta(6w)]^{2}\D w.\label{eq:IIKKKK_Eichler}
\end{align}\end{proposition}\begin{proof}By the Parseval--Plancherel theorem for Hankel transforms, we have\begin{align}
\int_0^\infty [I_0(t)]^{2}[K_0(t)]^4t\D t=\frac{1}{2}\int_0^\infty\left|\int_0^\infty J_{0}(\sqrt{u}t)I_0(t)[K_0(t)]^2t\D t\right|^2\D u.
\end{align}Here, for $ \tau=2-\frac{1}{6w}$, the modular parameter [cf.~\eqref{eq:theta_cubic_id} and \eqref{eq:theta_eta_x9}]\begin{align} u={}&x^2=-\left[\frac{\theta\left(1-\frac{1}{3w}\right)}{\theta\left(3-\frac{1}{w}\right)}\right]^{4}=\left[ 1-\frac{\theta\left( \frac{1}{3}-\frac{1}{9w} \right)}{\theta\left( 3-\frac{1}{w} \right)} \right]^3-1\notag\\={}&8\left[\frac{\eta(\tau)}{\eta(2\tau)}\right]^{3}\left[ \frac{\eta(6\tau)}{\eta(3\tau)} \right]^{9}-1=\left[\frac{\eta(6w)}{\eta(3w)}\right]^{3}\left[ \frac{\eta(w)}{\eta(2w)} \right]^{9}-1\end{align}satisfies [cf.~\eqref{eq:logWeberf2_deriv'} and \eqref{eq:logWeberf2_deriv''}]\begin{align}
\frac{\D u}{\D w}=-18\pi i\left[\frac{\eta(6w)}{\eta(3w)}\right]^{3}\left[ \frac{\eta(w)}{\eta(2w)} \right]^{9}\frac{ [\eta (3 w ) \eta (6w )]^3}{\eta (w ) \eta (2w )}=-18\pi i\frac{[\eta (w)]^8 [\eta (6 w)]^6}{[\eta (2w)]^{10}},\label{eq:uw_Jac}
\end{align}so \eqref{eq:IIKKKK_Eichler} follows immediately.\end{proof}\begin{proposition}[Eichler formulation of $ \JYM(6,0;1)$]\label{prop:JYM601}We have \begin{align}
\int_0^\infty [J_0(t)]^6t\D t={}&\frac{12}{\pi i}\int_{0}^{i\infty}[\eta(w)\eta(2w)\eta(3w)\eta(6w)]^{2}\D  w\notag\\{}&-\frac{6}{\pi i}\int_{\frac12}^{\frac{1}{2}+i\infty}[\eta(w)\eta(2w)\eta(3w)\eta(6w)]^{2}\D w.\label{eq:JJJJJJ_Eichler}
\end{align}\end{proposition}\begin{proof}Applying the arguments in the last proposition directly to \eqref{eq:p3_small} and \eqref{eq:p3_medium}, we obtain\begin{align}
\int_0^\infty [J_0(t)]^6t\D t={}&\frac{12}{\pi i}\int_{0}^{i\infty}[\eta(w)\eta(2w)\eta(3w)\eta(6w)]^{2}\D w\notag\\{}&+\frac{12}{\pi i}\int^{0+i0^{+}}_{\frac{1}{3}+i0^{+}}[\eta(w)\eta(2w)\eta(3w)\eta(6w)]^{2}(3w-1)^{2}\D w,
\end{align}where the second integral runs along the semi-circular path $ w=(1+e^{i\varphi})/6,\varphi\in(0,\pi)$.

Before arriving at the expression in \eqref{eq:JJJJJJ_Eichler}, we need to perform modular transformations on the last integral.

Towards this end, we recall from Chan--Zudilin \cite{ChanZudilin2010}  that the group $ \varGamma_0(6)_{+2}=\langle \varGamma_0(6),\widehat W_2\rangle$,  constructed by adjoining $ \widehat W_2=\frac{1}{\sqrt{2}}\left(\begin{smallmatrix}2&-1\\6&-2\end{smallmatrix}\right)$ to $ \varGamma_0(6)$, enjoys a Hauptmodul \begin{align}
X_{6,2}(z)=\left[ \frac{\eta(3z)\eta(6z)}{\eta(z)\eta(2z)} \right]^4\label{eq:X62_defn}
\end{align}   and a weight-2 modular form\begin{align}
Z_{6,2}(z)=\frac{[\eta(z)\eta(2z)]^{3}}{\eta(3z)\eta(6z)}.\label{eq:Z62_defn}
\end{align}With these notations, we see that $[\eta(z)\eta(2z)\eta(3z)\eta(6z)]^2=[Z_{6,2}(z)]^2 X_{6,2}(z) $ is a modular form of weight 4 on $ \varGamma_0(6)_{+2}$. In particular, we have \begin{align} [\eta(\widehat W_2z)\eta(2\widehat W_2z)\eta(3\widehat W_2z)\eta(6\widehat W_2z)]^2=4(3z-1)^{4}[\eta(z)\eta(2z)\eta(3z)\eta(6z)]^2.\end{align}Consequently, a variable substitution $ w=\widehat W_2z$ brings us \begin{align}&
\frac{12}{\pi i}\int^{0+i0^{+}}_{\frac{1}{3}+i0^{+}}[\eta(w)\eta(2w)\eta(3w)\eta(6w)]^{2}(3w-1)^{2}\D w\notag\\={}&-\frac{6}{\pi i}\int_{\frac12}^{\frac{1}{2}+i\infty}[\eta(z)\eta(2z)\eta(3z)\eta(6z)]^{2}\D z,
\end{align}thereby completing the proof.\end{proof}

David Broadhurst considered the following modular form of weight 4 and level 6\begin{align}
f_{4,6}(z)=[\eta(z)\eta(2z)\eta(3z)\eta(6z)]^2=\sum_{n=1}^\infty a_{4,6}(n)e^{2\pi inz},
\end{align} based on a suggestion from Francis Brown at Les Houches in 2010. Drawing on the work of Hulek \textit{et al.}~\cite{Hulek2001} that related the aforementioned modular form to a Kloosterman problem, Broadhurst conjectured that $  \IKM(2,4;1)$ is equal to $ \frac32 L(f_{4,6},3)$ \cite[][(110)]{Broadhurst2016}, where the special $L$-value can be written explicitly as \cite[][(108)]{Broadhurst2016}\begin{align}
L(f_{4,6},3):=\sum_{n=1}^\infty\frac{a_{4,6}(n)}{n^3}\left( 1+\frac{2\pi n}{\sqrt{6}}+\frac{2\pi^2 n^2}{3} \right)e^{-2\pi n/\sqrt{6}}.\label{eq:Lf46_3_defn}
\end{align}
 We now verify Broadhurst's conjecture.\begin{theorem}[$ \IKM(2,4;1)$ as a critical $L$-value]We have \begin{align}
 \IKM(2,4;1)=\int_0^\infty [I_0(t)]^{2}[K_0(t)]^4t\D t=\frac{3}{2}L(f_{4,6},3).
\end{align}\end{theorem}\begin{proof}Judging from termwise integration of uniformly  convergent series, we note that Broadhurst's conjecture essentially says that \begin{align}
\int_0^\infty [I_0(t)]^{2}[K_0(t)]^4t\D t=6\pi^{3}i\int_{i/\sqrt{6}}^{i\infty}f_{4,6}(w)\left( w^{2} -\frac{1}{6}\right)\D w.\label{eq:IIKKKK_L_value'}
\end{align}What we will do is to show that this statement is consistent with our results in  Propositions \ref{prop:IKM241} and \ref{prop:JYM601}.
Here, one can prove\begin{align}
6\pi^{3}i\int_{i/\sqrt{6}}^{i\infty}f_{4,6}(w)w^{2}\D w=-\pi^{3}i\int^{i/\sqrt{6}}_{0}f_{4,6}(z)\D z
\end{align}by a change of variable $ w=-1/{(6z)}$ and the  modular transformation $ \eta(-1/\tau)=\sqrt{\tau/i}\eta(\tau)$, so the right-hand side of \eqref{eq:IIKKKK_L_value'} is the same as $-\pi^{3}i\int_{0}^{i\infty}f_{4,6}(w)\D w. $
However, according to Propositions \ref{prop:IKM241} and \ref{prop:JYM601}, we have \begin{align}
-\pi^{3}i\int_{0}^{i\infty}f_{4,6}(w)\D w={}&\frac{\pi^4}{12}\int_0^\infty[J_0(t)]^6 t\D t-\frac{3}{2}\int_0^\infty [I_0(t)]^{2}[K_0(t)]^4t\D t.
\end{align}Meanwhile, the Wick rotation in \eqref{eq:IIKKKK_JJJJJJ} tells us that this is precisely $ \IKM(2,4;1)$, as conjectured by Broadhurst.
\end{proof}Before applying Proposition \ref{prop:Y_K_transforms} to  the 4-loop sunrise diagram $ \IKM(1,5;1)$, we need a cancelation formula related to Hankel and $ Y$-transforms.\begin{lemma}[Hilbert cancelation]\label{lm:HT_JY}Consider a   continuous function $ F(t),t>0$, whose Kramers--Kronig transform\begin{align}
(\widehat {\mathscr K}F)(\tau):=\int_{-\infty}^\infty\frac{F(|t|)|t|\D t}{\pi(\tau-t)},\quad \tau\in\mathfrak H
\end{align}is well-defined, and has the following asymptotic behavior:\begin{align}\begin{cases}
(\widehat {\mathscr K}F)(\tau)=O(\sqrt{\tau}),&\text{as }|\tau|\to0,\\(\widehat {\mathscr K}F)(\tau)=O\left( \dfrac{1}{|\tau|} \right),&\text{as }|\tau|\to\infty.\end{cases}\end{align} Suppose that  $ \int_0^\infty J_0(xt) F(t)t\D t,x\in(0,\infty)$ and   $ \int_0^\infty Y_0(xt) F(t)t\D t,x\in(0,\infty)$   are  both well-defined,  then \begin{align}\int_0^\infty \left[ \int_0^\infty J_0(xt) F(t)t\D t \right]\left[\int_0^\infty Y_0(x\tau) F(\tau)\tau\D \tau\right]x\D x
=0.
\end{align}\end{lemma}\begin{proof}According to the asymptotic behavior of $ \widehat {\mathscr K}F$, we have a vanishing identity for all $ x>0$:\begin{align}
\int_{i0^+-\infty}^{i0^++\infty}H_0^{(1)}(x\tau)(\widehat {\mathscr K}F)(\tau)\D \tau=0.\label{eq:H0KF_Jordan}
\end{align} Here, the contour can be closed upwards, thanks to Jordan's lemma. As $ \I \tau\to0^+$, we have the following Plemelj jump  relation for $ \xi\in(-\infty,0)\cup(0,\infty)$:\begin{align}
(\widehat {\mathscr K}F)(\xi+i0^{+})=\mathscr P\int_{-\infty}^\infty\frac{F(|t|)|t|\D t}{\pi(\xi-t)}-iF(|\xi|)|\xi|,
\end{align}where $ \mathscr P$ denotes Cauchy principal value. Here, the first term on the right-hand side of the equation above is the Hilbert transform of an even function $ F(|t|)|t|,t\in(-\infty,0)\cup(0,\infty)$, so it must be an odd function in $\xi$ \cite[][\S4.2]{KingVol1}. Meanwhile, we know that \begin{align}
H_0^{(1)}(x\xi+i0^{+})=\begin{cases}J_{0}(x\xi)+iY_0(x\xi),\ & \xi>0,\ \\
-J_{0}(x|\xi|)+iY_0(x|\xi|), & \xi<0,\ \\
\end{cases}
\end{align}so the vanishing identity in \eqref{eq:H0KF_Jordan} brings us \begin{align}
\int_0^\infty Y_0(xt)F(t)t\D t=-\int_0^\infty J_0(x\xi)\left[\mathscr  P\int_{-\infty}^\infty\frac{F(|t|)|t|\D t}{\pi(\xi-t)} \right]\D \xi.\label{eq:Hilbert_JY_pair}
\end{align}

Now we compute \begin{align}&
\int_0^\infty \left[ \int_0^\infty J_0(xt) F(t)t\D t \right]\left[\int_0^\infty Y_0(x\tau) F(\tau)\tau\D \tau\right]x\D x\notag\\={}&-\int_0^\infty \left[ \int_0^\infty J_0(xt) F(t)t\D t \right]\left\{\int_0^\infty J_0(x\tau) \left[\mathscr  P\int_{-\infty}^\infty\frac{F(|t|)|t|\D t}{\pi(\tau-t)} \right]\D \tau\right\}x\D x\notag\\={}&-\int_0^\infty F(\tau)\left[\mathscr  P\int_{-\infty}^\infty\frac{F(|t|)|t|\D t}{\pi(\tau-t)} \right]\D \tau=\frac{1}{4}\I\int_{i0^+-\infty}^{i0^++\infty}\frac{[(\widehat {\mathscr K}F)(\tau)]^{2}\D \tau}{\tau}.
\end{align}The last contour integral is indeed vanishing, because the integrand remains bounded as $ \tau\to i0^+$, and we can close the contour upwards, according to the asymptotic behavior of the Kramers--Kronig transform $ \widehat {\mathscr K}F$.\end{proof}\begin{theorem}[Sunrise at 4 loops]We have\begin{align}
\frac{3}{\pi^{2}}\int_0^\infty I_0(t)[K_0(t)]^5 t\D t=\int_0^\infty [I_0(t)]^{3}[K_0(t)]^3 t\D t={}&-6\pi^2\int_{0}^{i\infty}f_{4,6}(z)z\D z=\frac{3}{2}L(f_{4,6},2),\label{eq:sunrise_4loop}
\end{align}as stated in \eqref{eq:IKM151_331_eta_int}.\end{theorem}\begin{proof}The first equality in \eqref{eq:sunrise_4loop}
has been proved in \cite[][Lemma 3.1]{HB1}, as a special case ($ m=3,n=1$) of \eqref{eq:HB_sum_rule}. The last equality comes from the definition of $L$-functions via  Mellin transforms of cusp forms. The rest of this proof will revolve around the second equality.

We combine \eqref{eq:IKK_Hankel_mod} with \eqref{eq:JKKK_YIKK}, and carry out computations as in Proposition \ref{prop:IKM241}, to arrive at \begin{align}
\int_0^\infty I_0(t)[K_0(t)]^5 t\D t=\frac{\pi^{4}}{2}\int_{-\frac{1}{2}}^{-\frac{1}{2}+i\infty}f_{4,6}(w)(1+2w)\D w.
\end{align}Here, we have used the Parseval--Plancherel identity\begin{align}
\int_0^\infty \left\{ \int_0^\infty J_0(xt) I_{0}(t)[K_{0}(t)]^{2}t\D t \right\}\left\{\int_0^\infty J_0(x\tau)[K_{0}(\tau)]^{3}\tau\D \tau\right\}x\D x
=\int_0^\infty I_0(t)[K_0(t)]^5 t\D t
\end{align} and the Hilbert cancelation \begin{align}
\int_0^\infty \left\{ \int_0^\infty J_0(xt) I_{0}(t)[K_{0}(t)]^{2}t\D t \right\}\left\{\int_0^\infty Y_0(x\tau)  I_{0}(\tau)[K_{0}(\tau)]^{2}\tau\D \tau\right\}x\D x
=0.
\end{align}

By an analog of Proposition \ref{prop:JYM601}, we fuse \eqref{eq:p3_small}--\eqref{eq:p3_medium} and \eqref{eq:JJJY_YJJJ} together into the following formula:\begin{align}
\int_0^\infty [J_0(t)]^5 Y_{0}(t)t\D t={}&\frac{8}{\pi}
\int_{0}^{i\infty}f_{4,6}(w)w\D w+\frac{8}{\pi }\int^{0+i0^{+}}_{\frac{1}{3}+i0^{+}}f_{4,6}(w)w(1-3w)\D w.
\end{align} Again, a variable substitution $ w=\widehat W_2z$ gives rise to\begin{align}
\frac{8}{\pi }\int^{0+i0^{+}}_{\frac{1}{3}+i0^{+}}f_{4,6}(w)w(1-3w)\D w=\frac{4}{\pi}\int_{\frac{1}{2}}^{\frac{1}{2}+i\infty}f_{4,6}(z)(1-2z)\D z.
\end{align}

Thus, we have \begin{align}
\frac{8}{\pi}
\int_{0}^{i\infty}f_{4,6}(w)w\D w=\int_0^\infty [J_0(t)]^5 Y_{0}(t)t\D t+\frac{8}{\pi^{5}}\int_0^\infty I_0(t)[K_0(t)]^5 t\D t
\end{align} by cancelation of Eichler integrals. We can rewrite the equation above as\begin{align}
\frac{8}{\pi}
\int_{0}^{i\infty}f_{4,6}(w)w\D w=-\frac{4}{\pi^{5}}\int_0^\infty I_0(t)[K_0(t)]^5 t\D t,
\end{align}with the aid of \eqref{eq:IKKKKK_JJJJJY}. As we have  \cite[cf.][(107)]{Broadhurst2016}\begin{align}
&-6\pi^2\int_{0}^{i\infty}f_{4,6}(z)z\D z=-12\pi^2\int_{i/\sqrt{6}}^{i\infty}f_{4,6}(z)z\D z\notag\\={}&\frac{3}{2}\sum_{n=1}^\infty\frac{a_{4,6}(n)}{n^2}\left( 2+\frac{4\pi n}{\sqrt{6}} \right)e^{-2\pi n/\sqrt{6}}
\end{align}by termwise integration, this completes the proof. \end{proof}Like the determinant of \eqref{eq:BM_m2}, Broadhurst--Mellit also proposed that \cite[(113)]{Broadhurst2016}\begin{align}
\det
\begin{pmatrix}\IKM(1,5;1) & \IKM(1,5;3) \\
\IKM(2,4;1) & \IKM(2,4;3) \\
\end{pmatrix}=\frac{\pi^{4}}{2^63^2}.
\end{align}We have recently verified this conjecture in \cite[][\S3]{Zhou2017BMdet}, without explicitly computing individual matrix elements.

The Eichler integral representations for the first column in the determinant above have already been discussed. In a recent talk at the Erwin Schr\"odinger Institute \cite[][\S7.3]{Broadhurst2017ESI}, Broadhurst has announced his discoveries of representations for the second column as integrals over modular forms. We now prove Broadhurst's  empirical formulae.
\begin{theorem}[Broadhurst integrals for $ \IKM(1,5;3)$ and $ \IKM(2,4;3)$]Setting $v=3\left[ \frac{\eta(3z)}{\eta(z)} \right]^4\left[ \frac{\eta(2z)}{\eta(6z)} \right]^2 $ and $G(z)=f_{4,6}(z)(v^4-6v^2+2-6v^{-2}+9v^{-4})$, we have \begin{align}
\IKM(2,4;3)={}&\frac{\pi^3}{i}\int_{\frac12}^{\frac{1}{2}+i\infty}\frac{G(z)}{96}\D z,\label{eq:IKM243_mod}\\\IKM(1,5;3)={}&-3\pi^{4}\int_{\frac12}^{\frac{1}{2}+i\infty}\frac{G(z)}{96}\left( z-\frac{1}{2} \right)\D z.\label{eq:IKM153_mod}
\end{align}\end{theorem}\begin{proof}Writing $ f(u):=\int_0^\infty J_0(\sqrt{u}t)I_0(t)[K_0(t)]^2t\D t$ for $ u>0$, and using the Bessel differential equation along with  $\widehat B_3f(u)=0 $  [cf.~\eqref{eq:defn_B3}], one can show that \begin{align}\begin{split}&
\int_0^\infty J_0(\sqrt{u}t)I_0(t)[K_0(t)]^2t^{3}\D t\\={}&-4\left( u\frac{\D^2}{\D u^2}+\frac{\D }{\D u} \right)\int_0^\infty J_0(\sqrt{u}t)I_0(t)[K_0(t)]^2t\D t\\={}&\frac{4 [2 u (u+5) f'(u)+(u+3) f(u)]}{(u+1) (u+9)}.\end{split}
\end{align}  By Hankel fusion and integration by parts, we have \begin{align}\begin{split}&
\int_0^\infty [I_0(t)]^{2}[K_0(t)]^4t^{3}\D t\\={}&\frac{1}{2}\int_0^\infty\frac{4f(u) [2 u (u+5) f'(u)+(u+3) f(u)]}{(u+1) (u+9)} \D u\\={}&2\int_0^\infty \left[ \frac{1}{4 (u+1)}-\frac{1}{2 (u+1)^2}+\frac{3}{4 (u+9)}-\frac{9}{2 (u+9)^2} \right] [f(u)]^2\D u.\end{split}
\end{align}As we may recall from Proposition \ref{prop:IKM241}, the differential form $\frac{[f(u)]^2\D u}{2} $ translates into $ \frac{\pi^3i}{3}[\eta(z)\eta(2z)$ $\eta(3z)\eta(6z)]^{2}\D z$ for $\R z=-\frac12$, and \begin{align}
u+1=\left[\frac{\eta(6z)}{\eta(3z)}\right]^{3}\left[ \frac{\eta(z)}{\eta(2z)} \right]^{9},\quad u+9=9\frac{ \eta (z)}{\eta (6z)}\left[\frac{ \eta (3z)}{ \eta (2z)}\right]^5,
\end{align}so  $ \IKM(2,4;3)$ has  an  integral representation:\begin{align}
\begin{split}&\frac{\pi^{3}}{i}\int_{-\frac{1}{2}}^{-\frac{1}{2}+i\infty}
\left\{ \frac{2[ \eta (3 z)]^8 [\eta (2 z)]^{20}}{3 [\eta (z)]^{16} [\eta (6 z)]^4}+\frac{2[ \eta (6 z)]^4 [\eta (2 z)]^{12}}{27[ \eta (3 z)]^8}\right.\\{}&\left.-\frac{[\eta (3 z)]^5 [\eta (2 z)]^{11}}{3[ \eta (z)]^7 \eta (6 z)}-\frac{\eta (z)[ \eta (6 z)]^3 [\eta (2 z)]^7}{9[ \eta (3 z)]^3} \right\}\D z.\label{eq:IKM243_mod_prep}\end{split}\end{align}
Here, the path of integration can be shifted to $\R z=\frac12 $, by periodicity of the integrand. To identify the integrand inside the braces of \eqref{eq:IKM243_mod_prep} with $ G(z)/96$ in \eqref{eq:IKM243_mod}, simply compare their $q$-expansions up to sufficiently many terms  \cite[][Remark 1]{ChanZudilin2010}. This proves Broadhurst's integral representation for $\IKM(2,4;3)$ in  \eqref{eq:IKM243_mod}.

To verify \eqref{eq:IKM153_mod}, we start by rewriting \eqref{eq:JKKK_YIKK} as\begin{align}\begin{split}
g(u):={}&\int_0^\infty J_0(\sqrt{u}t)[K_0(t)]^3t\D t-\frac{3\pi}{2}\int_0^\infty Y_0(\sqrt{u}t)I_0(t)[K_0(t)]^2t\D t\\={}&\frac{\pi^{2}}{2}\frac{1}{3+u} \, _2F_1\left(\left.\begin{array}{c}
\frac{1}{3},\frac{2}{3} \\1 \\
\end{array}\right|1-\frac{u^{2} (9+u)}{(3+u)^3}\right
)\end{split}\end{align} and noting that $ \widehat B_3g(u)=0$. We can subsequently deduce Broadhurst's integral representation for $ \IKM(1,5;3)$ from Hankel fusion and a vanishing identity for $ F(t)=I_0(t)[K_0(t)]^2$:\begin{align}
\begin{split}&
\int_0^\infty \left[ \int_0^\infty J_0(xt) F(t)t\D t \right]\left[\int_0^\infty Y_0(x\tau) F(\tau)\tau^{3}\D \tau\right]x\D x
\\{}&+\int_0^\infty \left[ \int_0^\infty J_0(xt) F(t)t^{3}\D t \right]\left[\int_0^\infty Y_0(x\tau) F(\tau)\tau\D \tau\right]x\D x=0,\end{split}
\end{align} which is provable by  a modest variation on Lemma \ref{lm:HT_JY}.    \end{proof}\section{Feynman diagrams with 8 Bessel factors\label{sec:8Bessel}}
\subsection{Hankel transforms and Wick rotations\label{subsec:8Bessel_Hankel_Wick}}
We open this section by a confirmation  of Broadhurst's conjecture on  $ \IKM(2,6;1)$. \begin{theorem}[Eichler integral formulation  of $ \IKM(2,6;1)$]\label{thm:IKM_261}We have \begin{align}&
\int_0^\infty[I_0(t)]^2[K_0(t)]^6t\D t=\frac{\pi^{5}}{4i}\int_{0}^{i\infty}\left\{ \frac{ [\eta (2 z) \eta (3 z)]^{9}}{[\eta (z)\eta (6 z)]^{3}}+\frac{ [\eta ( z) \eta (6 z)]^{9}}{[\eta (2z)\eta (3 z)]^{3}} \right\}\D z.\label{eq:IKM261_Eichler}
\end{align} \end{theorem}
\begin{proof}
By  the Parseval--Plancherel theorem for Hankel transforms, we have\begin{align}
\int_0^\infty[I_0(t)]^2[K_0(t)]^6t\D t=\int_0^\infty\left|\int_0^\infty J_{0}(xt)I_0(t)[K_0(t)]^3t\D t\right|^2x\D x.\label{eq:IKM261_Hankel}
\end{align}  With the modular parametrization in \eqref{eq:IKKK_Hankel_eta}, and the Jacobian in \eqref{eq:xz_Jac},  we transition from an integration over the variable $x\in(0,\infty)$ to its counterpart over the variable $z$ on the $ \I z$-axis. Accordingly, we see that \begin{align}&\int_0^\infty[I_0(t)]^2[K_0(t)]^6t\D t\notag\\={}&\frac{\pi^{5}}{4i}\int_{0}^{i\infty}[\eta (z) \eta (2 z) \eta (3 z) \eta (6 z)]^{2}\left\{ \frac{[\eta (z ) \eta (2z )]^3}{\eta (3z ) \eta (6z )}+9\frac{ [\eta (3z ) \eta (6z )]^3}{\eta (z ) \eta (2z )} \right\}\D z\end{align} descends from \eqref{eq:IKM261_Hankel}.

Meanwhile, one can establish the following identity\begin{align}&
[\eta (z) \eta (2 z) \eta (3 z) \eta (6 z)]^{2}\left\{ \frac{[\eta (z ) \eta (2z )]^3}{\eta (3z ) \eta (6z )}+9\frac{ [\eta (3z ) \eta (6z )]^3}{\eta (z ) \eta (2z )} \right\}\notag\\={}&\frac{ [\eta (2 z) \eta (3 z)]^{9}}{[\eta (z)\eta (6 z)]^{3}}+\frac{ [\eta ( z) \eta (6 z)]^{9}}{[\eta (2z)\eta (3 z)]^{3}}\label{eq:f66_two_forms}
\end{align} by verifying that both sides are weight-6 modular forms on $ \varGamma_0(6)$, and checking the $q$-expansions of both sides up to sufficiently many terms \cite[][Remark 1]{ChanZudilin2010}.\end{proof}\begin{remark}Encouraged by Yun's recent contribution to Kloosterman sums \cite{Yun2015}, Broadhurst wrote \cite[][(135)]{Broadhurst2016}\begin{align}
f_{6,6}(z)=\frac{ [\eta (2 z) \eta (3 z)]^{9}}{[\eta (z)\eta (6 z)]^{3}}+\frac{ [\eta ( z) \eta (6 z)]^{9}}{[\eta (2z)\eta (3 z)]^{3}}=\sum_{n=1}^\infty a_{6,6}(n)e^{2\pi inz}
\end{align}and conjectured that $ \IKM(2,4;1)=\frac{27}{2}L(f_{6,6},5)$ for \cite[][(141) and (145)]{Broadhurst2016} \begin{align}
L(f_{6,6},5):=\sum_{n=1}^\infty \frac{a_{6,6}(n)}{n^{5}}\left( 1+\frac{2\pi n}{\sqrt{6}}+\frac{2\pi^2 n^2}{3} +\frac{2\pi^{3}}{9\sqrt{6}}+\frac{\pi^4n^4}{27}\right)e^{-2\pi n/\sqrt{6}}.
\end{align}This said the same thing as \begin{align}
\IKM(2,4;1)=\frac{9\pi^{5}}{i}\int_{i/\sqrt{6}}^{i\infty}\left\{ \frac{ [\eta (2 z) \eta (3 z)]^{9}}{[\eta (z)\eta (6 z)]^{3}}+\frac{ [\eta ( z) \eta (6 z)]^{9}}{[\eta (2z)\eta (3 z)]^{3}} \right\}\left( z^{4} +\frac{1}{36}\right)\D z,
\end{align}which is also equivalent to \eqref{eq:IKM261_Eichler} per a Fricke involution $ z\mapsto-1/(6z)$ and a modular transformation  $ \eta(-1/\tau)=\sqrt{\tau/i}\eta(\tau)$. \eor\end{remark}\begin{remark}In an earlier version of his conjecture, Broadhurst formulated the modular form $ f_{6,6} $ as \cite[][(90) and (91)]{Broadhurst2013MZV}\begin{align}
f_{6,6}(z)=[\eta (z) \eta (2 z) \eta (3 z) \eta (6 z)]^{2}\left( \sum_{m,n\in\mathbb Z}e^{2\pi i(m^2+mn+n^2)z} \right)\left( \sum_{m,n\in\mathbb Z}e^{4\pi i(m^2+mn+n^2)z} \right).
\end{align}This is of course compatible with the left-hand side of \eqref{eq:f66_two_forms}, in view of an identity by Borwein--Borwein--Garvan \cite[][Proposition 2.2(i)(ii) and Theorem 2.6(i)]{BorweinBorweinGarvan}.  \eor\end{remark}

Before handling other Bessel moments $ \IKM(a,b;1)$ satisfying $ a+b=8$, we need a modest generalization of Lemma \ref{lm:6BesselWick} and  modular parametrizations of some Hankel transforms not covered in \S\ref{sec:5Bessel}. \begin{lemma}[Some identities for Bessel moments]\begin{enumerate}[leftmargin=*,  label=\emph{(\alph*)},ref=(\alph*),
widest=a, align=left]\item The following formulae are true:\begin{align}\left(\frac2\pi\right)^6
\int_0^\infty[I_0(t)]^2[K_0(t)]^6t\D t={}&-\frac{8}{7}\int_0^\infty [J_0(x)]^6\{[J_0(x)]^2-7[Y_0(x)]^2\}x\D x,\label{eq:IKM261_Wick}\\\left(\frac2\pi\right)^4
\int_0^\infty[I_0(t)]^4[K_0(t)]^4t\D t={}&-\frac{4}{5}\int_0^\infty [J_0(x)]^6\{[J_0(x)]^2-5[Y_0(x)]^2\}x\D x.\label{eq:IKM441_Wick}\\\int_0^\infty [J_0(x)]^{4}[Y_{0}(x)]^4x\D x={}&-\frac{1}{5}\int_0^\infty [J_0(x)]^6\{[J_0(x)]^2-10[Y_0(x)]^2\}x\D x.\label{eq:JY_1_10_5}
\end{align}\item For $ x\in[0,2]$, we have\begin{align}
\int_0^\infty J_0(xt)[J_0(t)]^4t\D t=3\int_0^\infty J_0(xt)[J_0(t)]^{2}[Y_{0}(t)]^2t\D t.\label{eq:JJJJ_JJYY_Hankel_Wick}
\end{align}\item  For $ x\in[0,2]$, we have\begin{align}
\left(\frac2\pi\right)^3
\int_0^\infty I_{0}(xt)I_0(t)[K_0(t)]^3t\D t=-2\int_0^\infty J_0(xt)[J_0(t)]^{3}Y_{0}(t)t\D t.\label{eq:IIKKK_JJJJY_Wick}
\end{align} \end{enumerate}\end{lemma}\begin{proof}\begin{enumerate}[leftmargin=*,  label=(\alph*),ref=(\alph*),
widest=a, align=left]\item
By Wick rotation, we have\begin{align}
\left(\frac2\pi\right)^6
\int_0^\infty[I_0(t)]^2[K_0(t)]^6t\D t={}&\R\int_0^{i\infty}[J_0(z)]^2[H_0^{(1)}(z)]^6z\D z\notag\\={}&\R\int_0^\infty [J_0(x)]^2[H_0^{(1)}(x)]^6x\D x\notag\\={}& \int_0^\infty J^2( J^6-15 J^4 Y^2+15 J^2 Y^4-Y^6 ) x\D x,\label{eq:IKM261_Wick_prep}
\end{align}for $ J=J_0(x),Y=Y_0(x)$. With\begin{align}&
J^2( J^6-15 J^4 Y^2+15 J^2 Y^4-Y^6 )\notag\\{}&-\frac{J}{14}  [(J+i Y)^7-(-J+i Y)^7]-J^3 [(J+i Y)^5-(-J+i Y)^5]\notag\\={}&-\frac{8}{7}  J^6 (J^2-7 Y^2),
\end{align} we are able to  reduce \eqref{eq:IKM261_Wick_prep}  into \eqref{eq:IKM261_Wick}, by virtue of  \eqref{eq:JY_BHJ} in Lemma \ref{lm:BHJ}.

One can prove \eqref{eq:IKM441_Wick}  in a similar vein.

To prove \eqref{eq:JY_1_10_5}, compute\begin{align}
\frac{J^3}{2} [(J+i Y)^5-(-J+i Y)^5]=J^{4}(J^4-10J^2Y^2+5Y^4)
\end{align}and invoke   \eqref{eq:JY_BHJ}.
 \item By a variation on \eqref{eq:JJ_3YY}, we have the following vanishing identity when $ x\in[0,2]$:\begin{align}
\int_0^\infty J_0(xt)J\frac{(J+iY)^3-(-J+iY)^3}{2}t\D t=\int_0^\infty J_0(xt)J^{2}(J^{2}-3Y^{2})t\D t=0,
\end{align}with $ J=J_0(t),Y=Y_0(t)$.\item By Wick rotation, we can show that \begin{align}
\left(\frac2\pi\right)^3
\int_0^\infty I_{0}(xt)I_0(t)[K_0(t)]^3t\D t=-\int_0^\infty J_0(xt)(3J^{3}Y-JY^{3})t\D t,
\end{align}where  $ J=J_0(t),Y=Y_0(t)$. Meanwhile, when  $ x\in[0,4]$, we also have \begin{align}
\int_0^\infty J_0(xt)\frac{(J+iY)^4-(-J+iY)^4}{8i}t\D t=\int_0^\infty J_0(xt)(J^{3}Y-JY^{3})t\D t=0,
\end{align}  by an extension of Lemma \ref{lm:BHJ}. \qedhere\end{enumerate}\end{proof}\begin{proposition}[Hankel transforms related to $ \JYM$] \begin{enumerate}[leftmargin=*,  label=\emph{(\alph*)},ref=(\alph*),
widest=a, align=left]\item For   $ z=\frac12+\frac{i}{2\sqrt{3}}e^{i\varphi},\varphi\in(0,\pi/3)$, we have \begin{align}
\int_0^\infty J_{0}\left(i\left[\frac{2 \eta (2z ) \eta (6 z )}{\eta (z ) \eta (3z)}\right]^{3}t\right)[J_0(t)]^4t\D t=\frac{1-6z+12z^{2}}{4\pi i}Z_{6,3}(z)\label{eq:JJJJ_Hankel_high}
\end{align}where  $ x=i\left[\frac{2 \eta (2z ) \eta (6 z )}{\eta (z ) \eta (3z)}\right]^{3}$ maps $ \varphi\in(0,\pi/3)$ bijectively to $ x\in(2,4)$; for $ x\geq4$, we have \begin{align}
\int_0^\infty J_{0}(xt)[J_0(t)]^4t\D t=0.\label{eq:JJJJ_high_vanish}
\end{align}Consequently, we have\begin{align}&
\int_0^\infty[J_0(x)]^8x\D x\notag\\={}&\frac{36}{\pi i}\int_{\frac{1}{2}+\frac{i}{2\sqrt{3}}}^{\frac{1}{2}+i\infty}f_{6,6}(z)(1-2z)^{2}\D z+\frac{4}{\pi i}\int_{\frac{1}{4}+\frac{i}{4\sqrt{3}}}^{\frac{1}{2}+\frac{i}{2\sqrt{3}}}f_{6,6}(z)(1-6z+12z^{2})^{2}\D z.\label{eq:J8_Eichler}
\end{align}\item For $z=\frac12+iy,y\in\left(\frac{1}{2\sqrt{3}},\infty\right) $ and    $ z=\frac12+\frac{i}{2\sqrt{3}}e^{i\varphi},\varphi\in[0,\pi/3)$, the formula\begin{align}
\int_0^\infty J_{0}\left(i\left[\frac{2 \eta (2z ) \eta (6 z )}{\eta (z ) \eta (3z)}\right]^{3}t\right)[J_0(t)]^{2}[Y_{0}(t)]^2t\D t=\frac{2z-1}{4\pi i}Z_{6,3}(z)\label{eq:JJYY_Hankel}
\end{align} parametrizes $ \int_0^\infty J_0(xt)[J_0(t)]^{2}[Y_{0}(t)]^2t\D t$ for   $ x\in(0,4)$, and brings us\begin{align}&
\int_0^\infty [J_0(x)]^6[Y_0(x)]^2x\D x\notag\\={}&\frac{12}{\pi i}\int_{\frac{1}{2}+\frac{i}{2\sqrt{3}}}^{\frac{1}{2}+i\infty}f_{6,6}(z)(1-2z)^{2}\D z-\frac{4}{\pi i}\int_{\frac{1}{4}+\frac{i}{4\sqrt{3}}}^{\frac{1}{2}+\frac{i}{2\sqrt{3}}}f_{6,6}(z)(1-2z)(1-6z+12z^{2})\D z.\label{eq:J6Y2_Eichler}
\end{align}In addition, for $z=(1+e^{i\psi})/6,\psi\in[\pi/3,\pi) $, we have\begin{align}
\int_0^\infty J_{0}\left(i\left[\frac{2 \eta (2z ) \eta (6 z )}{\eta (z ) \eta (3z)}\right]^{3}t\right)[J_0(t)]^{2}[Y_{0}(t)]^2t\D t=-\frac{z(1-3z)}{\pi i}Z_{6,3}(z),\label{eq:JJYY_Hankel'}
\end{align}which parametrizes $ \int_0^\infty J_0(xt)[J_0(t)]^{2}[Y_{0}(t)]^2t\D t$ for   $ x\in[4,\infty)$ and leads us to \begin{align}&
\int_0^\infty [J_0(x)]^{4}[Y_0(x)]^4x\D x\notag\\={}&\frac{4}{\pi i}\int_{\frac{1}{4}+\frac{i}{4\sqrt{3}}}^{\frac{1}{2}+i\infty}f_{6,6}(z)(1-2z)^{2}\D z+\frac{64}{\pi i}\int_0^{\frac{1}{4}+\frac{i}{4\sqrt{3}}}f_{6,6}(z)z^{2}(1-3z)^{2}\D z.
\end{align} \item For  $z=\frac12+iy,y\in\left(\frac{1}{2\sqrt{3}},\infty\right) $, we have\begin{align}
\int_0^\infty J_{0}\left(i\left[\frac{2 \eta (2z ) \eta (6 z )}{\eta (z ) \eta (3z)}\right]^{3}t\right)[J_0(t)]^{3}Y_{0}(t)t\D t={}&-\frac{1}{4\pi }Z_{6,3}(z),\label{eq:J3Y_Hankel_S}
\intertext{which parametrizes $ \int_0^\infty J_{0}(xt)[J_0(t)]^{3}Y_{0}(t)t\D t$ for $ x\in(0,2)$; for     $ z=\frac12+\frac{i}{2\sqrt{3}}e^{i\varphi},\varphi\in[0,\pi/3)$, the identity}
\int_0^\infty J_{0}\left(i\left[\frac{2 \eta (2z ) \eta (6 z )}{\eta (z ) \eta (3z)}\right]^{3}t\right)[J_0(t)]^{3}Y_{0}(t)t\D t={}&\frac{1-6z+6z^{2}}{4\pi }Z_{6,3}(z)\label{eq:J3Y_Hankel_M}
\intertext{parametrizes $ \int_0^\infty J_{0}(xt)[J_0(t)]^{3}Y_{0}(t)t\D t$ for $ x\in[2,4)$; for $z=(1+e^{i\psi})/6,\psi\in[\pi/3,\pi) $, we have}\int_0^\infty J_{0}\left(i\left[\frac{2 \eta (2z ) \eta (6 z )}{\eta (z ) \eta (3z)}\right]^{3}t\right)[J_0(t)]^{3}Y_{0}(t)t\D t={}&-\frac{3z^2}{2\pi}Z_{6,3}(z),\label{eq:J3Y_Hankel_L}\end{align}a formula that parametrizes $ \int_0^\infty J_{0}(xt)[J_0(t)]^{3}Y_{0}(t)t\D t$ for $ x\in[4,\infty)$. As a result, the following identity holds:\begin{align}&
\int_0^\infty [J_0(x)]^6[Y_0(x)]^2x\D x\notag\\={}&-\frac{4}{\pi i}\int_{\frac{1}{2}+\frac{i}{2\sqrt{3}}}^{\frac{1}{2}+i\infty}f_{6,6}(z)\D z-\frac{4}{\pi i}\int_{\frac{1}{4}+\frac{i}{4\sqrt{3}}}^{\frac{1}{2}+\frac{i}{2\sqrt{3}}}f_{6,6}(z)(1-6z+6z^{2})^{2}\D z\notag\\{}&-\frac{144}{\pi i}\int_0^{\frac{1}{4}+\frac{i}{4\sqrt{3}}}f_{6,6}(z)z^{4}\D z.\label{eq:J6Y2_Eichler'}
\end{align}\end{enumerate}\end{proposition}\begin{proof}

\begin{enumerate}[leftmargin=*,  label=(\alph*),ref=(\alph*),
widest=a, align=left]\item \label{itm:JJJJ_Hankel}Judging from \eqref{eq:fx_Sym2}, we know that \begin{align}
\int_0^\infty J_{0}(xt)[J_0(t)]^4t\D t=Z_{6,3}(z)(c_{0}+c_1z+c_2 z^2),\quad x\in(2,4),
\end{align}where the constants $ c_0$, $c_1$ and $c_2$ can be determined by the continuity at $x=2$ and the asymptotic behavior as $ x\to4^-$ \cite[][Theorem 4.1]{BSWZ2012}.
This
proves \eqref{eq:JJJJ_Hankel_high}.

To show \eqref{eq:JJJJ_high_vanish}, read off the real part from the following Wick rotation:\begin{align}
\int_0^\infty H_{0}^{(1)}(xt)[J_0(t)]^4t\D t=\frac{2i}{\pi}\int_0^\infty [I_0(t)]^4K_{0}(xt)t\D t,\quad \forall x\geq4.
\end{align}

Applying the Parseval--Plancherel theorem for Hankel transforms to \eqref{eq:JJJJ_Hankel_eta}  and \eqref{eq:JJJJ_Hankel_high}, we arrive at \eqref{eq:J8_Eichler}.\item For $ z=\frac12+iy,y\in\left(\frac{1}{2\sqrt{3}},\infty\right)$, the Hankel transform formula in \eqref{eq:JJYY_Hankel} follows from  \eqref{eq:JJJJ_Hankel_eta}  and \eqref{eq:JJJJ_JJYY_Hankel_Wick}. The remaining arguments run parallel to those in \ref{itm:JJJJ_Hankel}.  \item To verify \eqref{eq:J3Y_Hankel_S}, simply combine \eqref{eq:IKKK_I} with \eqref{eq:IIKKK_JJJJY_Wick}. The rest founds on similar principles as the proof of \ref{itm:JJJJ_Hankel}.   \qedhere\end{enumerate}
\end{proof}\begin{remark}We note that Borwein \textit{et al.}~expressed $ \int_0^\infty J_{0}(xt)[J_0(t)]^4t\D t,x\in(2,4)$ as generalized hypergeometric series \cite[][Theorem 4.7]{BSWZ2012}, but did not give a modular parametrization. \eor\end{remark}
\begin{proposition}[$Y$- and $K$-transforms]For $ z=\frac12+iy,y\in\left(\frac{1}{2\sqrt{3}},\infty\right) $, we have\begin{align}&
\int_0^\infty I_{0}\left(i\left[\frac{2 \eta (2z ) \eta (6 z )}{\eta (z ) \eta (3z)}\right]^{3}t\right)[K_0(t)]^{4}t\D t+4\int_0^\infty K_{0}\left(i\left[\frac{2 \eta (2z ) \eta (6 z )}{\eta (z ) \eta (3z)}\right]^{3}t\right)I_{0}(t)[K_0(t)]^{3}t\D t\notag\\={}&\frac{\pi^3(2z-1)}{8i}Z_{6,3}(z).\label{eq:IKKKK_KIKKK}
\end{align}For $ z/i>0$, we have \begin{align}&
\int_0^\infty J_{0}\left(\left[\frac{2 \eta (2z ) \eta (6 z )}{\eta (z ) \eta (3z)}\right]^{3}t\right)[K_0(t)]^{4}t\D t-2\pi\int_0^\infty Y_{0}\left(\left[\frac{2 \eta (2z ) \eta (6 z )}{\eta (z ) \eta (3z)}\right]^{3}t\right)I_{0}(t)[K_0(t)]^{3}t\D t\notag\\={}&\frac{\pi^3z}{4i}Z_{6,3}(z).\label{eq:JKKKK_YIKKK}
\end{align}\end{proposition}\begin{proof}Let $ \widehat A_4$ be the Picard--Fuchs operator given in \eqref{eq:op_A4}, then one can verify \begin{align}
\widehat A_4\left\{ \int_0^\infty K_0(xt)I_0(t)[K_0(t)]^3xt\D t \right\}=6x^{3}
\end{align}by differentiation under the integral sign, and integration by parts \cite[cf.][\S9]{Vanhove2014Survey}. Comparing this to  \eqref{eq:A4_inhom}, we know that $ \int_0^\infty I_0(xt)[K_0(t)]^4xt\D t+4 \int_0^\infty K_0(xt)I_0(t)[K_0(t)]^3xt\D t$ is annihilated by $ \widehat A_4$. Therefore, the left-hand side of \eqref{eq:IKKKK_KIKKK} must assume the form\begin{align}
Z_{6,3}(z)\left[ k_0 +k_1(2z-1)+k_2(2z-1)^2\right],
\end{align}for certain  constants $ k_0$, $k_1$, and $k_2$. Since $ K_0(xt)=-\log(xt)+O(1)$ as $ x\to0^+$, and $ \int_0^\infty I_0(t)[K_0(t)]^3t\D t=\pi^2/16$ \cite[][(54)]{Bailey2008}, the left-hand side of \eqref{eq:IKKKK_KIKKK} behaves  like $ \frac{\pi^3(2z-1+o(z))}{8i}Z_{6,3}(z)$ as $ z\to \frac12+i\infty$. This shows that $ k_1=\frac{\pi^3}{8i}$ and $k_2=0$. To demonstrate that  $ k_0=0$, simply check the special value at $ z=\frac12+\frac{i\sqrt{5}}{2\sqrt{3}}$ against Theorem \ref{thm:Bologna} and Table \ref{tab:spec_X63_Z63}.

As we perform analytic continuation on the left-hand side of \eqref{eq:IKKKK_KIKKK} to the positive $ \I z$-axis, and extract the real part, we arrive at \eqref{eq:JKKKK_YIKKK}.  \end{proof}\begin{remark}From a Hilbert transform formula \cite[cf.][(3.2)]{HB1}\begin{align}
\mathscr P\int_{-\infty}^\infty\frac{2\pi I_0(t)[K_0(|t|)]^3|t|\D t}{\pi(\tau-t)}=\{[\pi I_0(\tau)]^2-[K_0(|\tau|)]^2\}[K_0(|\tau|)]^2\tau,\quad \forall\tau\in\mathbb R\smallsetminus\{0\},
\end{align}we can deduce [cf.~\eqref{eq:Hilbert_JY_pair}]\begin{align}
\int_0^\infty J_0(xt)\{[\pi I_0(t)]^2-[K_0(t)]^2\}[K_0(t)]^2t\D t=-2\pi\int_0^\infty Y_0(xt)I_0(t)[K_0(t)]^3t\D t,\quad \forall x>0.
\end{align} Thus, we may recast \eqref{eq:JKKKK_YIKKK} into \begin{align}
\int_0^\infty J_{0}\left(\left[\frac{2 \eta (2z ) \eta (6 z )}{\eta (z ) \eta (3z)}\right]^{3}t\right)[I_{0}(t)]^{2}[K_0(t)]^{2}t\D t={}&\frac{\pi z}{4i}Z_{6,3}(z)\label{eq:JKKKK_YIKKK'}\tag{\ref{eq:JKKKK_YIKKK}$'$}
\end{align}for $z/i>0$.\eor\end{remark}\begin{remark}From \eqref{eq:IKKK_KIKK} and \eqref{eq:IKKKK_KIKKK}, we see that when  $ n$ is 3 or 4, and $ u\in(0,1)$, the expression\begin{align}
\int_0^\infty I_0(\sqrt{u}t)[K_0(t)]^{n}t\D t+n\int_0^\infty K_0(\sqrt{u}t)I_{0}(t)[K_0(t)]^{n-1}t\D t
\end{align}is annihilated by a differential operator (in $u$) of order $n-1$. The same pattern actually applies to all $n\in\mathbb Z_{\geq2}$, and the corresponding differential operator has been constructed by Vanhove in \cite[][\S9]{Vanhove2014Survey}. The steps of integrations by parts leading to these homogeneous differential equations are described in \cite[][Lemma 4.2]{Zhou2017BMdet}. Such homogeneous differential equations are crucial in our recent proofs \cite[][\S4]{Zhou2017BMdet} of two determinant formulae proposed by Broadhurst--Mellit \cite[][Conjectures 4 and 7]{Broadhurst2016}.    \eor\end{remark}\subsection{Critical $L$-values for Bessel moments}
A conjectural sum rule    $ 9\pi^2\IKM(4,4;1)-14\IKM(2,$ $6;1)=0$ dated back to 2008 \cite[][at the end of \S6.3, between (228) and (229)]{Bailey2008}, and was restated as an open problem in 2016 \cite[][(147)]{Broadhurst2016}. It has also been conjectured that  \cite[][(139) and (143)]{Broadhurst2016} \begin{align}
\int_0^\infty[I_0(t)]^4[K_0(t)]^4t\D t=L(f_{6,6},3):=\sum_{n=1}^\infty \frac{a_{6,6}(n)}{n^{3}}\left( 2+\frac{4\pi n}{\sqrt{6}}+\frac{2\pi^2 n^2}{3} \right)e^{-2\pi n/\sqrt{6}}.
\end{align} With the preparations in \S\ref{subsec:8Bessel_Hankel_Wick}, we can verify these claims.\begin{theorem}[Relation between  $ \IKM(2,6;1)$ and $ \IKM(4,4;1)$]\begin{enumerate}[leftmargin=*,  label=\emph{(\alph*)},ref=(\alph*),
widest=a, align=left]\item We have a vanishing identity\begin{align}
\int^{\frac12+i\infty}_{\frac{1}{2}+\frac{i}{2\sqrt{3}}}f_{6,6}(z)(1-2z)^{2}\D z+\int_{\frac{1}{4}+\frac{i}{4\sqrt{3}}}^{\frac{1}{2}+\frac{i}{2\sqrt{3}}}f_{6,6}(z)(1 - 4 z + 8 z^2)\D z=0.
\label{eq:2arc_sum0}\end{align}\item We have a sum rule\begin{align}
9\pi^2\int_0^\infty[I_0(t)]^4[K_0(t)]^4t\D t-14\int_0^\infty[I_0(t)]^2[K_0(t)]^6t\D t=0.\label{eq:IKM441_IKM261_sum}
\end{align}\end{enumerate}\end{theorem}\begin{proof}
\begin{enumerate}[leftmargin=*,  label=(\alph*),ref=(\alph*),
widest=a, align=left]\item We  spell out both sides of  \eqref{eq:JY_1_10_5} using Hankel fusions. The left-hand side becomes\begin{align}
\int_0^\infty [J_0(x)]^{4}[Y_0(x)]^4x\D x=\frac{20}{\pi i}\int_{\frac{1}{2}+\frac{i}{2\sqrt{3}}}^{\frac{1}{2}+i\infty}f_{6,6}(z)(1-2z)^{2}\D z+\frac{4}{\pi i}\int_{\frac{1}{4}+\frac{i}{4\sqrt{3}}}^{\frac{1}{2}+\frac{i}{2\sqrt{3}}}f_{6,6}(z)(1-2z)^{2}\D z,\label{eq:J4Y4_fusion}
\end{align}where
we have transformed
\begin{align}\int_0^{\frac{1}{4}+\frac{i}{4\sqrt{3}}}f_{6,6}(z)z^{2}(1-3z)^{2}\D z={}&\frac{1}{4}\int^{-\frac12+i\infty}_{-\frac{1}{2}+\frac{i}{2\sqrt{3}}}f_{6,6}(z)(1+2z)^{2}\D z\notag\\={}&\frac{1}{4}\int^{\frac12+i\infty}_{\frac{1}{2}+\frac{i}{2\sqrt{3}}}f_{6,6}(z)(1-2z)^{2}\D z,\end{align} by a Fricke involution $ z\mapsto-1/(6z)$ and a horizontal translation. The right-hand side becomes\begin{align}&
-\frac{1}{5}\int_0^\infty [J_0(x)]^6\{[J_0(x)]^2-10[Y_0(x)]^2\}x\D x\notag\\={}&\frac{84}{5\pi i}\int^{\frac12+i\infty}_{\frac{1}{2}+\frac{i}{2\sqrt{3}}}f_{6,6}(z)(1-2z)^{2}\D z-\frac{4}{5\pi i}\int_{\frac{1}{4}+\frac{i}{4\sqrt{3}}}^{\frac{1}{2}+\frac{i}{2\sqrt{3}}}f_{6,6}(z)\left[\frac{11}{9} + \frac{4}{3} \left(z - \frac{1}{3}\right)\right.\notag\\{}&\left. + 12\left(z - \frac{1}{3}\right) ^2 - 192\left(z - \frac{1}{3}\right) ^3 +
 144\left(z - \frac{1}{3}\right) ^4\right]\D z,\label{eq:J4Y4_fusion'_prep}
\end{align}according to \eqref{eq:IKM261_Wick},  \eqref{eq:J8_Eichler}, and \eqref{eq:J6Y2_Eichler}. We bear in mind that $ f_{6,6}(z)=[Z_{6,2}(z)]^3X_{6,2}(z)[1+9X_{6,2}(z)]$ is a modular form of weight 6 on $ \varGamma_0(6)_{+2}$, which transforms under $ \widehat W_2 z=\frac{2z-1}{6z-2}$ as \begin{align}
f_{6,6}(\widehat W_2z)=-8(3z-1)^{6}f_{6,6}(z).\label{eq:f66_W2}
\end{align}Thus, the identities\begin{align}
144\int_{\frac{1}{4}+\frac{i}{4\sqrt{3}}}^{\frac{1}{2}+\frac{i}{2\sqrt{3}}}f_{6,6}(z)\left(z -\frac{1}{3}\right)  ^4\D z={}&\frac{4}{9}\int_{\frac{1}{4}+\frac{i}{4\sqrt{3}}}^{\frac{1}{2}+\frac{i}{2\sqrt{3}}}f_{6,6}(z)\D z,\\192\int_{\frac{1}{4}+\frac{i}{4\sqrt{3}}}^{\frac{1}{2}+\frac{i}{2\sqrt{3}}}f_{6,6}(z)\left(z -\frac{1}{3}\right)  ^3\D z={}&-\frac{32}{3}\int_{\frac{1}{4}+\frac{i}{4\sqrt{3}}}^{\frac{1}{2}+\frac{i}{2\sqrt{3}}}f_{6,6}(z)\left(z -\frac{1}{3}\right)  \D z
\end{align}allow us to rewrite \eqref{eq:J4Y4_fusion'_prep} as \begin{align}&
-\frac{1}{5}\int_0^\infty [J_0(x)]^6\{[J_0(x)]^2-10[Y_0(x)]^2\}x\D x\notag\\={}&\frac{84}{5\pi i}\int^{\frac12+i\infty}_{\frac{1}{2}+\frac{i}{2\sqrt{3}}}f_{6,6}(z)(1-2z)^{2}\D z+\frac{4}{5\pi i}\int_{\frac{1}{4}+\frac{i}{4\sqrt{3}}}^{\frac{1}{2}+\frac{i}{2\sqrt{3}}}f_{6,6}(z)(1-4z-12z^{2})\D z.\label{eq:J4Y4_fusion'}
\end{align}

Identifying  \eqref{eq:J4Y4_fusion} with \eqref{eq:J4Y4_fusion'},
we arrive at  \eqref{eq:2arc_sum0}, as claimed.\item In the light of \eqref{eq:IKM261_Wick} and \eqref{eq:IKM441_Wick}, we see that the proposed sum rule is equivalent to the following vanishing identity:\begin{align}
\int_0^\infty[J_0(x)]^6\{2[J_0(x)]^2-5[Y_0(x)]^2\}x\D x=0.
\end{align}

We may compute\begin{align}&
\int_0^\infty[J_0(x)]^6\{2[J_0(x)]^2-5[Y_0(x)]^2\}x\D x\notag\\={}&\frac{12}{\pi i}\int^{\frac12+i\infty}_{\frac{1}{2}+\frac{i}{2\sqrt{3}}}f_{6,6}(z)(1-2z)^{2}\D z+\frac{1}{\pi i}\int_{\frac{1}{4}+\frac{i}{4\sqrt{3}}}^{\frac{1}{2}+\frac{i}{2\sqrt{3}}}f_{6,6}(z)\left[ \frac{28}{9} + \frac{32}{3} \left(z - \frac{1}{3}\right)\right.\notag\\{}&\left. + 96\left(z - \frac{1}{3}\right)^2 - 96\left(z - \frac{1}{3}\right)^3 +
 1152\left(z - \frac{1}{3}\right)^4 \right]\D z\notag\\={}&\frac{12}{\pi i}\int^{\frac12+i\infty}_{\frac{1}{2}+\frac{i}{2\sqrt{3}}}f_{6,6}(z)(1-2z)^{2}\D z+\frac{12}{\pi i}\int_{\frac{1}{4}+\frac{i}{4\sqrt{3}}}^{\frac{1}{2}+\frac{i}{2\sqrt{3}}}f_{6,6}(z)(1-4z+8z^2)\D z=0,
\end{align}where the first equality comes  from  \eqref{eq:J8_Eichler} and \eqref{eq:J6Y2_Eichler}, while the second and third equalities hinge on  \eqref{eq:f66_W2} and \eqref{eq:2arc_sum0}, respectively. \qedhere\end{enumerate}\end{proof}
\begin{theorem}[Relation between $ L(f_{6,6},3) $ and $ L(f_{6,6},5) $]\begin{enumerate}[leftmargin=*,  label=\emph{(\alph*)},ref=(\alph*),
widest=a, align=left]\item We have \begin{align}&
\frac{7}{6\pi^{5}i}\int_0^\infty[I_0(t)]^2[K_0(t)]^6t\D t\notag\\={}&\int_{\frac{1}{2}+\frac{i}{2\sqrt{3}}}^{\frac{1}{2}+i\infty}f_{6,6}(z)(1-2z^{2})\D z+\int_{\frac{1}{4}+\frac{i}{4\sqrt{3}}}^{\frac{1}{2}+\frac{i}{2\sqrt{3}}}f_{6,6}(z)z^{2}\D z-2 \int_0^{\frac{1}{4}+\frac{i}{4\sqrt{3}}}f_{6,6}(z)z^{2}\D z.\label{eq:I2K6_id1}
\end{align} \item We have \begin{align}&
\frac{21}{2\pi^{5}i}\int_0^\infty[I_0(t)]^2[K_0(t)]^6t\D t\notag\\={}&\int_{\frac{1}{2}+\frac{i}{2\sqrt{3}}}^{\frac{1}{2}+i\infty}f_{6,6}(z)(2+17z^{2})\D z+23\int_{\frac{1}{4}+\frac{i}{4\sqrt{3}}}^{\frac{1}{2}+\frac{i}{2\sqrt{3}}}f_{6,6}(z)z^{2}\D z+17 \int_0^{\frac{1}{4}+\frac{i}{4\sqrt{3}}}f_{6,6}(z)z^{2}\D z.\label{eq:I2K6_id2}
\end{align} \item The following integral identity holds:\begin{align}
\int_{0}^{i\infty}f_{6,6}(z)z^4\D z+\frac{2}{7}\int_{0}^{i\infty}f_{6,6}(z)z^{2}\D z=0,\label{eq:4over7_EShM_prep}
\end{align}which implies\begin{align}
\frac{L(f_{6,6},5)}{\zeta(2)L(f_{6,6},3)}=\frac{4}{7},\label{eq:4over7_EShM}
\end{align}where \begin{align}
L(f_{6,6},3):=\sum_{n=1}^\infty \frac{a_{6,6}(n)}{n^{3}}\left( 2+\frac{4\pi n}{\sqrt{6}}+\frac{2\pi^2 n^2}{3} \right)e^{-2\pi n/\sqrt{6}}.\label{eq:Lf66_3}
\end{align}\end{enumerate}\end{theorem}\begin{proof}
\begin{enumerate}[leftmargin=*,  label=(\alph*),ref=(\alph*),
widest=a, align=left]\item\label{itm:IKM261_a} According to \eqref{eq:IKM261_Wick},  \eqref{eq:J8_Eichler},  \eqref{eq:J6Y2_Eichler} and   \eqref{eq:f66_W2}, we have \begin{align}&
-\frac{7}{8}\left(\frac2\pi\right)^6
\int_0^\infty[I_0(t)]^2[K_0(t)]^6t\D t+\frac{48}{\pi i}\int_{\frac{1}{2}+\frac{i}{2\sqrt{3}}}^{\frac{1}{2}+i\infty}f_{6,6}(z)(1-2z)^{2}\D z\notag\\={}&\frac{4}{\pi i}\int_{\frac{1}{4}+\frac{i}{4\sqrt{3}}}^{\frac{1}{2}+\frac{i}{2\sqrt{3}}}f_{6,6}(z)\left[ \frac{8}{9} +\frac{4}{3}  \left(z -\frac{1}{3}\right)  + 12\left(z -\frac{1}{3}\right)  ^2 - 120\left(z -\frac{1}{3}\right)  ^3 +
 144\left(z -\frac{1}{3}\right)  ^4 \right]\D z\notag\\={}&\frac{48}{\pi i}\int_{\frac{1}{4}+\frac{i}{4\sqrt{3}}}^{\frac{1}{2}+\frac{i}{2\sqrt{3}}}f_{6,6}(z)z^{2}\D z.\label{eq:I2K6_prep0}
\end{align}

In the meantime, by complex conjugation, we have\begin{align}
\int_{\frac{1}{2}+\frac{i}{2\sqrt{3}}}^{\frac{1}{2}+i\infty}f_{6,6}(z)(1-2z)^{2}\D z=\int_{-\frac{1}{2}+\frac{i}{2\sqrt{3}}}^{-\frac{1}{2}+i\infty}f_{6,6}(z)(1+2z)^{2}\D z,
\end{align}whereas $ f_{6,6}(z)=f_{6,6}(z+1)$ brings us\begin{align}
\int_{\frac{1}{2}+\frac{i}{2\sqrt{3}}}^{\frac{1}{2}+i\infty}f_{6,6}(z)(1-4z)\D z+\int_{-\frac{1}{2}+\frac{i}{2\sqrt{3}}}^{-\frac{1}{2}+i\infty}f_{6,6}(z)(1+4z)\D z=-2\int_{\frac{1}{2}+\frac{i}{2\sqrt{3}}}^{\frac{1}{2}+i\infty}f_{6,6}(z)\D z.
\end{align}Therefore, we obtain\begin{align}{}&
\int_{\frac{1}{2}+\frac{i}{2\sqrt{3}}}^{\frac{1}{2}+i\infty}f_{6,6}(z)(1-2z)^{2}\D z\notag\\={}&-\int_{\frac{1}{2}+\frac{i}{2\sqrt{3}}}^{\frac{1}{2}+i\infty}f_{6,6}(z)\D z+2\int_{\frac{1}{2}+\frac{i}{2\sqrt{3}}}^{\frac{1}{2}+i\infty}f_{6,6}(z)z^{2}\D z+2\int_{-\frac{1}{2}+\frac{i}{2\sqrt{3}}}^{-\frac{1}{2}+i\infty}f_{6,6}(z)z^{2}\D z\notag\\={}&-\int_{\frac{1}{2}+\frac{i}{2\sqrt{3}}}^{\frac{1}{2}+i\infty}f_{6,6}(z)\D z+2\left( \int_0^{\frac{1}{4}+\frac{i}{4\sqrt{3}}}+\int_{\frac{1}{2}+\frac{i}{2\sqrt{3}}}^{\frac{1}{2}+i\infty}\right)f_{6,6}(z)z^{2}\D z,
\end{align}after invoking  $ f_{6,6}(-1/(6z))=-216z^{6}f_{6,6}(z)$ in the last step.

All this allows us to rearrange \eqref{eq:I2K6_prep0} into \eqref{eq:I2K6_id1}.\item\label{itm:IKM261_b} In view of  \eqref{eq:IKM261_Wick},  \eqref{eq:J8_Eichler}, and \eqref{eq:J6Y2_Eichler'}, we have\begin{align}&
-\frac{7}{8}\left(\frac2\pi\right)^6
\int_0^\infty[I_0(t)]^2[K_0(t)]^6t\D t-\frac{4}{\pi i}\int_{\frac{1}{2}+\frac{i}{2\sqrt{3}}}^{\frac{1}{2}+i\infty}f_{6,6}(z)[9(1-2z)^{2}+7]\D z\notag\\={}&\frac{4}{\pi i}\int_{\frac{1}{4}+\frac{i}{4\sqrt{3}}}^{\frac{1}{2}+\frac{i}{2\sqrt{3}}}f_{6,6}(z)\left[ \frac{8}{9} +\frac{32}{3}  \left(z -\frac{1}{3}\right)  + 12\left(z -\frac{1}{3}\right)  ^2 - 120\left(z -\frac{1}{3}\right)  ^3 +
 396\left(z -\frac{1}{3}\right)  ^4 \right]\D z\notag\\{}&+\frac{1008}{\pi i}\int_0^{\frac{1}{4}+\frac{i}{4\sqrt{3}}}f_{6,6}(z)z^{4}\D z.\label{eq:I2K6_prep1}
\end{align}  As before, we may reduce\begin{align}&
\int_{\frac{1}{2}+\frac{i}{2\sqrt{3}}}^{\frac{1}{2}+i\infty}f_{6,6}(z)[9(1-2z)^{2}+7]\D z\notag\\={}&-2\int_{\frac{1}{2}+\frac{i}{2\sqrt{3}}}^{\frac{1}{2}+i\infty}f_{6,6}(z)\D z+18\left( \int_0^{\frac{1}{4}+\frac{i}{4\sqrt{3}}}+\int_{\frac{1}{2}+\frac{i}{2\sqrt{3}}}^{\frac{1}{2}+i\infty}\right)f_{6,6}(z)z^{2}\D z,\end{align}\begin{align}36\int_0^{\frac{1}{4}+\frac{i}{4\sqrt{3}}}f_{6,6}(z)z^{4}\D z={}&\int_{\frac{1}{2}+\frac{i}{2\sqrt{3}}}^{\frac{1}{2}+i\infty}f_{6,6}(z)\D z,
\end{align}and\begin{align}&
\int_{\frac{1}{4}+\frac{i}{4\sqrt{3}}}^{\frac{1}{2}+\frac{i}{2\sqrt{3}}}f_{6,6}(z)\left[ \frac{8}{9} +\frac{32}{3}  \left(z -\frac{1}{3}\right)  + 12\left(z -\frac{1}{3}\right)  ^2 - 120\left(z -\frac{1}{3}\right)  ^3 +
 396\left(z -\frac{1}{3}\right)  ^4 \right]\D z\notag\\={}&\frac13\int_{\frac{1}{4}+\frac{i}{4\sqrt{3}}}^{\frac{1}{2}+\frac{i}{2\sqrt{3}}}f_{6,6}(z)(-7 + 28 z + 36 z^2)\D z.\label{eq:quadratic_mid_arc}
\end{align}
 By virtue of the vanishing identity in  \eqref{eq:2arc_sum0}, the right-hand side of \eqref{eq:quadratic_mid_arc} is also equal to\begin{align}&
\frac{7}{3}\int^{\frac12+i\infty}_{\frac{1}{2}+\frac{i}{2\sqrt{3}}}f_{6,6}(z)(1-2z)^{2}\D z+\frac{92}{3}\int_{\frac{1}{4}+\frac{i}{4\sqrt{3}}}^{\frac{1}{2}+\frac{i}{2\sqrt{3}}}f_{6,6}(z)z^2\D z\notag\\={}&-\frac{7}{3}\int^{\frac12+i\infty}_{\frac{1}{2}+\frac{i}{2\sqrt{3}}}f_{6,6}(z)\D z+\frac{14}{3}\left( \int_0^{\frac{1}{4}+\frac{i}{4\sqrt{3}}}+\int_{\frac{1}{2}+\frac{i}{2\sqrt{3}}}^{\frac{1}{2}+i\infty}\right)f_{6,6}(z)z^{2}\D z+\frac{92}{3}\int_{\frac{1}{4}+\frac{i}{4\sqrt{3}}}^{\frac{1}{2}+\frac{i}{2\sqrt{3}}}f_{6,6}(z)z^2\D z.
\end{align}

Gathering the results above, we arrive at \eqref{eq:I2K6_id2}.\item Eliminating \begin{align}\int_{\frac{1}{2}+\frac{i}{2\sqrt{3}}}^{\frac{1}{2}+i\infty}f_{6,6}(z)\D z \end{align} from  \eqref{eq:I2K6_id1} and \eqref{eq:I2K6_id2}, we obtain\begin{align}
\int_{0}^{i\infty}f_{6,6}(z)z^{2}\D z=\frac{7}{18\pi^{5}i}\int_0^\infty[I_0(t)]^2[K_0(t)]^6t\D t,
\end{align}which is equivalent to \eqref{eq:4over7_EShM_prep}. [There is also an alternative way to arrive at the equation above, namely, by fusing \eqref{eq:JKKKK_YIKKK'} with itself, and referring to \eqref{eq:IKM441_IKM261_sum}.] Checking the definition of $ L(f_{6,6},3)$ in \eqref{eq:Lf66_3} against  termwise integration on the right-hand side of the following equation:\begin{align}
4\pi^{3}i\int_{0}^{i\infty}f_{6,6}(z)z^{2}\D z=8\pi^{3 }i\int_{i/\sqrt{6}}^{i\infty}f_{6,6}(z)z^{2}\D z,
\end{align} we can verify \eqref{eq:4over7_EShM}.
\qedhere\end{enumerate}\end{proof}\begin{remark}Previously, Broadhurst observed that $ L(f_{6,6},5)/[\zeta(2)L(f_{6,6},3)]$ must be a rational number, according to Eichler--Shimura--Manin theory \cite[cf.][Theorem 1]{Shimura1977}, and found this rational number to be numerically $4/7$ \cite[][(142)]{Broadhurst2016}.\eor\end{remark}\begin{remark}As a by-product of the foregoing computations, one may eliminate $ \JYM(6,2;1)$ from  \eqref{eq:IKM261_Wick} and \eqref{eq:IKM441_Wick}, to deduce\begin{align}
\int_0^\infty[J_0(x)]^8x\D x=\frac{70}{9\pi i }\int_{0}^{i\infty}f_{6,6}(z)\D z=-\frac{80}{\pi i}\int_{0}^{i\infty}f_{6,6}(z)z^{2}\D z=\frac{280}{\pi i}\int_{0}^{i\infty}f_{6,6}(z)z^{4}\D z,
\end{align}which gives $L$-series representations for a ``random walk integral'' $ \JYM(8,0;1)$.

Furthermore, we have recently shown \cite[][Theorem 5.1]{Zhou2017PlanarWalks} that for each $ j\in\mathbb Z_{>1}$, the function  $\int _{0}^{\infty}J_{0}(x t)[J_0(t)]^{2j+1}t\D t,0\leq x\leq 1 $ is a  $ \mathbb Q$-linear combination of \begin{align}
\int_0^\infty I_0(xt)[I_0(t)]^{2m+1}\left[\frac{K_{0}(t)}{\pi}\right]^{2(j-m)}t\D t,\;\text{where } m\in\mathbb Z\cap[0,(j-1)/2].\label{eq:pn_IKM}
\end{align}   This implies that, for all $ n\in\mathbb Z_{>4}$, the ``random walk integral''  $ \JYM(n,0;1)$  is a   $ \mathbb Q$-linear combination of $ \IKM(a,b;1)/\pi^b$ for  certain positive integers $ a$ and $b$ satisfying $ a+b=n  $.\eor\end{remark}

Finally, we verify Broadhurst's conjectures regarding $ \IKM(1,7;1)$ and $ \IKM(3,5;1)$.\begin{theorem}[Sunrise at 6 loops]We have \begin{align}
\pi^{2}\int_0^\infty [I_0(t)]^{3}[K_0(t)]^5t\D t=\int_0^\infty I_0(t)[K_0(t)]^7t\D t=-\pi^{6}\int_0^{i\infty} f_{6,6}(z)z\D z,\label{eq:sunrise_6loop}
\end{align}which is equivalent to  \eqref{eq:IKM171_351_eta_int}.\end{theorem}\begin{proof}The first equality in  \eqref{eq:sunrise_6loop}, which says\begin{align}
\int_0^\infty \frac{[\pi I_{0}(t)+iK_0(t)]^4-[\pi I_{0}(t)-iK_0(t)]^4}{i}[K_0(t)]^4t\D t=0,
\end{align}is a special case ($ m=4,n=1$) of \eqref{eq:HB_sum_rule'}.

Fusing together   \eqref{eq:IKKK_Hankel_eta} and \eqref{eq:JKKKK_YIKKK}, while noting that (see Lemma \ref{lm:HT_JY})\begin{align}
\int_0^\infty \left\{ \int_{0}^\infty J_0(xt)I_0(t)[K_0(t)]^3 t\D t \right\}\left\{ \int_{0}^\infty Y_0(x\tau)I_0(\tau)[K_0(\tau)]^3 \tau\D \tau \right\}x\D x=0,
\end{align}we arrive at the last equality in \eqref{eq:sunrise_6loop}, after  some computations similar to those in Theorem \ref{thm:IKM_261}. Alternatively, we can throw    \eqref{eq:IKKK_Hankel_eta} and \eqref{eq:JKKKK_YIKKK'} into the Parseval--Plancherel theorem for Hankel transforms, and invoke the first equality in   \eqref{eq:sunrise_6loop}.

It is clear that    \eqref{eq:sunrise_6loop} is compatible with \eqref{eq:IKM171_351_eta_int}, up to a Fricke involution $z\mapsto -1/(6z) $ in the integrand.       \end{proof}
\subsection*{Acknowledgments}This research was supported in part  by the Applied Mathematics Program within the Department of Energy
(DOE) Office of Advanced Scientific Computing Research (ASCR) as part of the Collaboratory on
Mathematics for Mesoscopic Modeling of Materials (CM4). 

This manuscript grew out of my research notes formerly intended for a project on automorphic representations \cite{AGF_PartI,EZF,AGF_PartII} at Princeton in 2013, and was completed in  2017 during my stay  in Beijing arranged by Prof.\ Weinan E (Princeton University and Peking University). I thank Dr.\ David Broadhurst for providing valuable  background information in quantum field theory, as well as for sharing with me his slides for  recent talks in Paris \cite{Broadhurst2017Paris}, Marseille \cite{Broadhurst2017CIRM}, Edinburgh \cite{Broadhurst2017Higgs}, Zeuthen \cite{Broadhurst2017DESY} and Vienna \cite{Broadhurst2017ESIa,Broadhurst2017ESI}.
I also thank him for pointing out an error in an early  draft.

I am deeply indebted to the anonymous referees for their careful examinations and detailed analyses of this work. I am grateful to them for their thoughtful comments and suggestions that helped me improve the presentation of this paper.

\vspace{2em}

\end{document}